\newif\iffinal
\finaltrue  % Final version

\documentclass[letterpaper,11pt,reqno]{amsart} 
\RequirePackage[utf8]{inputenc}
\usepackage[portrait,margin=3.25cm]{geometry}
\iffinal\else\usepackage[notref,notcite]{showkeys}\fi
\usepackage{mathrsfs}
\usepackage[colorlinks,citecolor=blue,urlcolor=blue,linkcolor=blue,pagecolor=blue,linktocpage=true,backref=true]%
{hyperref}
\usepackage{amssymb,amsthm,amsfonts,amsbsy,latexsym}
\usepackage{graphicx}
\usepackage[usenames,dvipsnames]{color}
\usepackage[numeric,initials,nobysame]{amsrefs}
\usepackage{upref,setspace}
\usepackage[below]{placeins}
\usepackage{enumerate}

\newenvironment{enumeratea}{\begin{enumerate}[\upshape (a)]}{\end{enumerate}}

\numberwithin{equation}{section}
\numberwithin{figure}{section}
\numberwithin{table}{section}

\renewcommand{\leq}{\leqslant} 
\renewcommand{\geq}{\geqslant}

\newcommand{\eps}{\varepsilon}

\newcommand{\set}[1]{\left\{#1\right\}}

\def\qed{ \hfill $\blacksquare$}  
\newcommand{\cE}{\mathcal{E}}
\newcommand{\cG}{\mathcal{G}}\newcommand{\cH}{\mathcal{H}}\newcommand{\cI}{\mathcal{I}}

\newcommand{\cT}{\mathcal{T}}

\newcommand{\bR}{\mathbb{R}}

\newtheorem{theorem}{Theorem}[section]
\newtheorem{corollary}[theorem]{Corollary}
\newtheorem{lemma}[theorem]{Lemma}
\newtheorem{proposition}[theorem]{Proposition}

\theoremstyle{definition}
\newtheorem{definition}[theorem]{Definition}
\theoremstyle{remark}
\newtheorem{remark}[theorem]{Remark}
\numberwithin{equation}{section}
\newcommand{\no}{\nonumber}
\newcommand{\be}{\begin{equation}}
\newcommand{\ee}{\end{equation}}
\newcommand{\bi}{\begin{itemize}}
\newcommand{\ei}{\end{itemize}}
\newcommand{\br}{\begin{eqnarray}}
\newcommand{\er}{\end{eqnarray}}
\newcommand{\Rm}{{\mathbb R}}

\newcommand{\ProbP}{\mathbb{P}}
\newcommand{\ProbQ}{\mathbb{Q}}
\newcommand{\Mean}{\mathbb{E}}
\newcommand{\indc}{{\bf1}}

\newcommand{\calE}{\mathcal{E}}
\newcommand{\calF}{\mathcal{F}}
\newcommand{\calG}{\mathcal{G}}
\newcommand{\calH}{\mathcal{H}}
\newcommand{\calI}{\mathcal{I}}
\newcommand{\calO}{\mathcal{O}}

\newcommand{\calT}{\mathcal{T}}
\newcommand{\calW}{\mathcal{W}}

\newcommand{\argsup}[1]{\arg\!\sup_{#1}}
\newcommand{\arginf}[1]{\arg\!\inf_{#1}}

\newcommand{\commentout}[1]{}
\newcommand{\ER}{Erd\H{o}s-R\'enyi }

\begin{document}

\title[IS for rare events in Erd\H{o}s-R\'enyi graphs]{The importance sampling technique for understanding rare events in Erd\H{o}s-R\'enyi random graphs}

\date{}
% Secondary: 62G32, 60F05, 60G70;}

\author[Bhamidi]{Shankar Bhamidi$^1$}
\address{$^1$Department of Statistics and Operations Research, 304 Hanes Hall, University of North Carolina, Chapel Hill, NC 27599}
\author[Hannig]{Jan Hannig$^2$}
\address{$^2$Department of Statistics and Operations Research, 330 Hanes Hall, University of North Carolina, Chapel Hill, NC 27599}
\author[Lee]{Chia Ying Lee$^3$}
\address{$^3$Statistical and Applied Mathematical Sciences Institute, 19 T.W. Alexander Drive, P.O. Box 14006,Research Triangle Park, NC 27709, USA.}
\author[Nolen]{James Nolen$^4$}
\address{Mathematics Department, Duke University, Box 90320, Durham, North Carolina, 27708, USA}
\email{bhamidi@email.unc.edu, hannig@email.unc.edu, leecy@email.unc.edu, nolen@math.duke.edu }

\maketitle

\begin{abstract}
  In dense \ER random graphs, we are interested in the events where large numbers of a given subgraph occur. The mean behavior of subgraph counts is known, and only recently were the related large deviations results discovered. Consequently, it is natural to ask, can one develop efficient numerical schemes to estimate the probability of an \ER graph containing an excessively large number of a fixed given subgraph? Using the large deviation principle we study an importance sampling scheme as a method to numerically compute the small probabilities of large triangle counts occurring within \ER graphs. We show that the exponential tilt suggested directly by the large deviation principle does not always yield an optimal scheme.
The exponential tilt used in the importance sampling scheme comes from a generalized class of exponential random graphs. Asymptotic optimality, a measure of the efficiency of the importance sampling scheme, is achieved by a special choice of the parameters in the exponential random graph that makes it indistinguishable from an \ER graph conditioned to have many triangles in the large network limit. We show how this choice can be made for the conditioned \ER graphs both in the replica symmetric phase as well as in parts of the replica breaking phase to yield asymptotically optimal numerical schemes to estimate this rare event probability. 
\end{abstract}

\section{Introduction}

In this paper we study the use of importance sampling schemes to numerically estimate the probability that an \ER random graph contains an unusually large number of triangles. A simple graph $X$ on $n$ vertices can be represented as an element of the space $\Omega_n = \set{0,1}^{\binom{n}{2}}$. A graph $X \in \Omega_n$ will be denoted by $X = (X_{ij})_{1 \leq i < j \leq n}$ with the entry $X_{ij}$ indicating the presence or absence of an edge between vertices $i$ and $j$.   For a given edge probability $p\in [0,1]$, an \ER random graph $\calG_{n,p}$ is a graph on $n$ vertices such that any edge is independently connected with probability $p$.
We shall use $\ProbP_{n,p}$ to represent the probability measure on $\Omega_n$ induced by the \ER graph $\calG_{n,p}$.   The probability of a fixed graph $X$ under the measure $\ProbP_{n,p}$ can be explicitly computed as
\begin{equation} \label{eq:Pnpdefn}
  \ProbP(G_{n,p}=X) = \ProbP_{n,p} (X) \, = \prod_{i<j} p^{X_{ij}} (1-p)^{1-X_{ij}} 
  \,= \, (1 - p)^{\binom{n}{2}}\,e^{h_p E(X)}
\end{equation} 
where 
\begin{equation} \label{eq:h_p}
  h_p := \log \frac{p}{1-p}
\end{equation}
and $E(X) := \sum_{i<j} X_{ij}$ is the number of edges in $X$. Let $T(X)$ denote the number of triangles in graph $X$:
\[
T(X) = \sum_{1 \leq i<j<k \leq n} X_{ij} X_{jk} X_{ik}.
\]
Also let the event $W_{n,t} = \set{X \in \Omega_n \;| \;\; T(X)\geq \binom{n}{3}t^3}$ denote the upper tails of triangle counts.
 Consider an \ER random graph $\calG_{n,p}$. For $p$ fixed, one can show that $\Mean[T(\cG_{n,p})] \sim {n \choose 3} p^3 $ as $n \to \infty$.  For $t>p$, the main aim of this paper is the following question: can one develop efficient numerical schemes to estimate the probability 
\be
\mu_{n} = \ProbP \left( T(\calG_{n,p}) \geq \binom{n}{3} t^3 \right) \label{smallprob1}
\ee
that $\calG_{n,p}$ has an atypically large number of triangles? Before addressing such questions, one first needs to understand the structure of such random graphs, conditioned on this rare event, more precisely the large deviation rate function for such events. The last few years have witnessed a number of deep results in understanding such questions including upper tails of triangle counts, along with more general subgraph densities (see e.g., \cites{Borgsetal08, Cha12, ChaDey10, ChaVar11, DeMKahn12, KimVu04, LubZhao12}). 
In the dense graph case, where the edge probability $p$ stays fixed as $n \rightarrow \infty$, \cite{ChaDey10} derived a large deviation principle (LDP) for the rare event $\{ T(\calG_{n,p}) \geq \textstyle{\binom{n}{3}} t^3 \}$, showing that for $t$ within a certain subset of $(p,1]$, 
\begin{equation} 
  \ProbP \left( T(\calG_{n,p}) \geq \binom{n}{3} t^3 \right) = \exp\left({-n^2 I_p(t)(1 + O(n^{-1/2})) }\right) \label{CDldp}
\end{equation}
where the rate function $I_p(t)$ is given by
\begin{equation} \label{Iprate}
I_p(t) = \frac{1}{2}\left(t \log \frac{t}{p} + (1 - t) \log \frac{1 - t}{1 - p} \right).
\end{equation}
More recently \cite{ChaVar11} showed a general large deviation principle for dense \ER graphs, using the theory of limits of dense random graph sequences developed recently by Lovasz et al. \cite{Lovasz2012monograph, LovSze07, LovSze06, Borgsetal08}.
When specialized to upper tails of triangle counts, they show that there exists a rate function $\phi(p,t)$
\begin{equation}
\label{eqn:cha-varadhan}
    \frac{1}{n^2} \log{\ProbP \left( T(\calG_{n,p}) \geq \binom{n}{3} t^3 \right)} \to -\phi(p,t), \qquad \mbox{ as } n\to \infty.
\end{equation}
The function $\phi(p,t)$ coincides with $I_p(t)$ for a certain parameter range of $(p,t)$, and is described in more detail in \eqref{eq: LDP rate runction}. The exponential decay of the probability of the event of interest makes it difficult to estimate this probability even for moderately large $n$. Direct Monte Carlo sampling is obviously intractable.
The central strategy of importance sampling is to sample from a different probability measure, the \emph{tilted measure}, under which the event of interest is no longer rare; one obtains more successful samples falling in the event of interest but each sample must then be weighted appropriately according to the Radon-Nikodym derivative of the original measure against the tilted measure.
Importance sampling techniques have been used in many other stochastic systems, such as SDEs and Markov processes and queuing systems, see e.g \cite{Buck,rubino2009rare,juneja2006rare,dupuis2004importance,blanchet2008efficient} and the references therein.
In particular, when a large deviations principle is known for the stochastic system, the tilted measure commonly used is a change of measure arising from the LDP.
However, not every tilted measure associated with the LDP works well.
It is well known that a poorly chosen tilted measure can lead to an estimator that performs worse than Monte Carlo sampling, or whose variance blows up \cite{GlasWang97}.
Thus, a careful choice of tilted measure is of utmost importance.  Before describing the relevant tilts we formally define our aims. 

\subsection{Importance sampling and asymptotic optimality} 
\label{sec:imp-samp-opt}
If $\{X^k\}_{k=1}^\infty \subset \Omega_n$ is a sequence of \ER random graphs generated independently from $\ProbP_{n,p}$, then for any integer $K \geq 1$,
\[
M_K =  \frac{1}{K} \sum_{k=1}^K \indc_{W_{n,t}}(X^k) 
\]
is an unbiased estimate of $\mu_{n}$. By the law of large numbers, $M_K \to \mu_{n}$ with probability one as $K \to \infty$. Although this estimate of $\mu_{n}$ is very simple, the relative error is
\[
\frac{\sqrt{\text{Var}(M_K)}}{\Mean(M_K)} = \frac{\sqrt{\mu_{n} - (\mu_{n})^2}}{\mu_{n} \sqrt{K}},
\]
which scales like $(K \mu_{n})^{-1/2}$ as $\mu_{n} \to 0$. Hence the relative error may be very large in the large deviation regime where $\mu_{n} <\!< 1$, unless we have at least $K \sim O(\mu_{n}^{-1})$ samples. Therefore, it is desirable to devise an estimate of $\mu_{n}$ which, compared to this simple Monte Carlo estimate, attains the same accuracy with fewer number of samples or lower computational cost.

Importance sampling is a Monte Carlo algorithm based on a change of measure. Suppose that $\ProbP_{n,p}$ is absolutely continuous with respect to another measure $\ProbQ$ on $\Omega_n$ with
\[
\frac{d\ProbP_{n,p}}{d \ProbQ} = Y^{-1} : \Omega_n \to \Rm.
\]
Then we have
\be
\mu_{n} = \Mean[M_K] =  \Mean \left[ \frac{1}{K} \sum_{k=1}^K \indc_{W_{n,t}}(X^k) \right] = \Mean_{\ProbQ} \left[ \frac{1}{K} \sum_{k=1}^K \indc_{W_{n,t}}(\tilde X^k) Y^{-1}(\tilde X^k) \right] \label{eq: measchange}
\ee
where $\Mean_{\ProbQ}$ denotes expectation with respect to $\ProbQ$, and we now use $\{\tilde X^k\}_{k=1}^\infty$ to denote a set of random graphs sampled independently from the new measure $\ProbQ$.  If we define
\be
\tilde M_K = \frac{1}{K} \sum_{k=1}^K \indc_{W_{n,t}}(\tilde X^k) Y^{-1}(\tilde X^k),  \label{Mkimp}
\ee
then $\tilde M_K$ is also an unbiased estimate of $\mu_{n}$, and the relative error is now:
\be
\frac{\sqrt{\text{Var}_{\ProbQ}(M_K)}}{\Mean_{\ProbQ}(M_K)} =  \frac{\sqrt{\Mean_{\ProbQ}[(\indc_{W_{n,t}}(X)Y^{-1})^2 ]- (\mu_{n})^2}}{\mu_{n} \sqrt{K}}, \label{eq: relative error}
\ee
Formally this is optimized by the choice $Y = (\mu_{n})^{-1} \indc_{W_{n,t}}(X)$, in which case the relative error is zero. Such a choice for $\ProbQ$ is not feasible, however, since normalizing $Y$ would require a priori knowledge of $\mu_{n} = \ProbP_{n,p}(W_{n,t})$. 
Intuitively, we should choose the tilted measure $\ProbQ$ so that $\tilde X_k \in W_{n,t}$ occurs with high probability under $\ProbQ$.

We will refer to $Y^{-1}$ as the importance sampling weights, and $\ProbQ$ as the {\em tilted measure}, or tilt. If $\ProbQ$ arises naturally as the measure induced by a random graph $\calG_{n}$, we will also refer to $\calG_{n}$ as the tilt. In the cases where a large deviation principle holds, it gives us an estimate of the relative error in the estimate $\tilde M_K$. For any fixed $K$, it is clear from (\ref{eq: relative error}) that minimizing the relative error is equivalent to minimizing the second moment $\Mean_{\ProbQ_{n}}[(\indc_{\calW_t}Y^{-1})^2]$.
Since Jensen's inequality implies that 
\begin{equation*}
  \Mean_{\ProbQ_{n}} [(\indc_{\calW_t}Y^{-1})^2] 
  \geq 
  \Mean_{\ProbQ_{n}} [\indc_{\calW_t}Y^{-1}]^2,
\end{equation*}
we have the following asymptotic lower bound:
\begin{equation} \label{eq: Jensen's asymptotic lower bound}
  \liminf_{n\rightarrow\infty} \frac{1}{n^2} \log \Mean_{\ProbQ_{n}} [(\indc_{\calW_t}Y^{-1})^2] \geq  - 2\phi(p,t).
\end{equation}
Thus, the presence of a large deviation principle for the random graphs $\calG_{n,p}$ as $n \to \infty$, leads to a way to quantify the efficiency of the importance scheme in an asymptotic sense, as is done in other contexts \cite{Buck}. 

\begin{definition} \label{def:asympEff}
 A family of tilted measures $\ProbQ_n$ on $\calW$ is said to be {\em asymptotically optimal} if 
\begin{equation*}
  \lim_{n\rightarrow\infty} \frac{1}{n^2} \log \Mean_{\ProbQ_{n}} [(\indc_{\calW_t}Y^{-1})^2] = - 2 \inf_{f\in\calW_t} [\mathcal{I}_p(f)].
\end{equation*}
\end{definition}
In contrast, the second moment of each term in the simple Monte Carlo method satisfies
\begin{equation*}
  \lim_{n\rightarrow\infty} \frac{1}{n^2} \log \Mean_{\ProbP_{n}} [\indc_{\calW_t}^2]
 = -\phi(p,t) > -2\phi(p,t).
\end{equation*}
Thus, the simple Monte Carlo method is not asymptotically optimal.
Observe that Jensen's inequality for conditional expectation implies
\br
\ProbQ_n(\calW_t)^{-1} & = &  \ProbP_n(\calW_t)^{-1} \left( \frac{\Mean_{\ProbP_{n}}( \indc_{\calW_t} Y )}{\ProbP_n(\calW_t)}\right)^{-1} \no \\
& \leq & \ProbP_n(\calW_t)^{-2} \Mean_{\ProbP_{n}}( \indc_{\calW_t} Y^{-1}  ) =  \ProbP_n(\calW_t)^{-2} \Mean_{\ProbQ_{n}}( \indc_{\calW_t} Y^{-2}  ).
\er
So, if $\ProbQ_n$ is asymptotically optimal, we must have
\br
\liminf_{n \to \infty} \frac{1}{n^2} \log \ProbQ_n(\calW_t) \geq \liminf_{n \to \infty} \frac{2}{n^2} \log \ProbP_n(\calW_t) + \liminf_{n \to \infty} \frac{-1}{n^2} \log \Mean_{\ProbQ_{n}}( \indc_{\calW_t} Y^{-2}  ) = 0,
\er
which is consistent with the intuition that a good choice of $\ProbQ_n$ should put $\tilde X_k \in \calW_t$ with high probability. 

To understand in this context the tilts that could be relevant, let us now describe in a little more detail, properties of the rate function \ref{eqn:cha-varadhan} as well as structural results of the \ER model conditioned on rare events and their connections to a sub-family of the famous exponential random graph models.

\subsection{Edge and triangle tilts}
\label{sec:intro-edge-tilt}

In this article we consider tilted measures within a family of exponential random graphs $\calG_{n}^{h,\beta,\alpha}$. For parameters $h\in \mathbb{R}$, $\beta \geq 0$, and $\alpha > 0$, these exponential random graphs are defined via the Gibbs measure, $\ProbQ_n = \ProbQ_{n}^{h,\beta,\alpha}$ on the space of simple graphs on $n$ vertices, where
\be
\ProbQ_n^{h,\beta,\alpha} (X) \propto e^{H(X)}, \qquad
\text{where }
H(X) = hE(X) + \frac{\beta}{n} \left(\frac{n^3}{6}\right)^{1-\alpha} T(X)^{\alpha}. \label{expGibbsdef}
\ee
$E(X)$ is the number of edges in graph $X$.  If $\beta = 0$, and $h = h_q = \log \frac{q}{1-q}$ for some $q \in (0,1)$, then $\calG_{n}^{h_q,0,\alpha}$ is an \ER graph with edge probability $q$ (notice that $\alpha$ is irrelevant when $\beta = 0$). In particular, the original graph $\calG_{n,p}$ is an exponential random graph with parameters $h = h_p$ and $\beta = 0$.

Given the rare event problem, $\calG_{n,p}$ conditioned on $T(\calG_{n,p})\geq \binom{n}{3}t^3$, which we shall henceforth parameterize by $(p,t)$, we will focus on two strategies for choosing the tilted measure. The first is to set $\beta = 0$ and $h = h_q$ for some $q > p$. The resulting tilted measure $\ProbQ_{n}^{h_q,0,\alpha}$ will be called an \emph{edge tilt}; compared to the original measure for $\calG_{n,p}$, this tilt simply puts more weight on edges. The second strategy is to set $h = h_p$ but vary $\beta > 0$ and $\alpha > 0$.  We refer to the resulting tilted measure $\ProbQ_{n}^{h_p,\beta,\alpha}$ as a \emph{triangle tilt}; compared to the original measure, this tilt puts more weight on triangles, while leaving $h = h_p$ unchanged.

That the two tilts above are natural candidates for the importance sampling scheme, can be reasoned in light of the following concept of when two graphs are alike. In \cite{ChaVar11} it is shown that for the range of $(p,t)$ where one has \eqref{CDldp}, the \ER graph $\calG_{n,p}$ conditioned on the rare event $\{T(\calG_{n,p}) \geq \binom{n}{3} t^3\}$ is \emph{asymptotically indistinguishable} from another \ER graph $\calG_{n,t}$ with edge probability $t$, in the sense that the typical graphs in the conditioned \ER graph resembles a typical graph drawn from $\calG_{n,t}$ when $n$ is large. (Asymptotic indistinguishability is explained more precisely at \eqref{eq: asymp indistinguishable}.) Thus, choosing the tilted measure to resemble the typical conditioned graph is a natural choice. While it may seem plausible for any $t > p$ that the conditioned graph resembles another \ER graph, since $\Mean [T(\calG_{n,t})] \sim \binom{n}{3} t^3$ as $n \to \infty$, it is not always the case. Depending on $p$ and $t$, it may be that the graph $\calG_{n,p}$ conditioned on the event $\{T(\calG_{n,p}) \geq \binom{n}{3} t^3\}$ tends to form cliques and hence does not resemble an \ER graph. When the conditioned graph does resemble an \ER graph, we say that $(p,t)$ is in the \emph{replica symmetric} phase. On the other hand, when the conditioned graph is not asymptotically indistinguishable from an \ER graph we say that $(p,t)$ is in the \emph{replica breaking} phase. (See Definition \ref{def:replica symmetric}.)

The main question we wish to address is: given the parameters $(p,t)$ for the rare event problem, how can we choose the tilt parameters ($h_q$ for the edge tilt, or $\beta$ and $\alpha$ for the triangle tilt) so that the resulting importance sampling scheme is asymptotically optimal? And, can an optimal importance sampling scheme be constructed for all values of $(p,t)$?

Regarding the edge tilt, our first result (Prop \ref{prop:edgeNo}) is that the edge tilt $\ProbQ_{n}^{h_q,0,\alpha}$ can be asymptotically optimal only if $h_q = h_t$ (i.e. $q = t$). This is not very surprising since $\Mean[T(\calG_{n,q})] \sim {n \choose 3} q^3$.  On the other hand, we also will prove, in Proposition \ref{prop:edgeNoOpt}, the more surprising result that for some values of $(p,t)$ the importance sampling scheme based on the edge tilt $\ProbQ_{n}^{h_t,0,\alpha}$ will not be asymptotically optimal. In particular, there is a subregime of the replica symmetric phase for which the edge tilt with $h = h_t$ produces a suboptimal estimator.

%%%%%%%%%%%%%%%%%%%
\commentout{
Intuitively, the tilted measure for an asymptotically optimal importance sampling scheme should closely resemble the original measure conditioned on the rare event of interest; that is, if the sampling scheme is to be asymptotically optimal, we should expect that samples from the tilted measure resemble the original graph $\calG_{n,p}$ conditioned to have at least ${n \choose 3}t^3$ triangles. If we choose $h = h_t$, the edge tilt $\ProbQ_{n}^{h_t,0,\alpha}$ is not precluded from being asymptotically optimal for those values of $(p,t)$ where the original graph $\calG_{n,p}$ conditioned to have at least ${n \choose 3}t^3$ triangles resembles the tilted graph $\calG_{n,t}$, a consequence of Prop \ref{prop:edgeNo}.
}%%%%%%%%%%%%%end commentout

Regarding the triangle tilt $\ProbQ_{n}^{h_p,\beta,\alpha}$, our main result (Prop \ref{prop: asymptotic optimality}) is a necessary and sufficient condition on the tilt parameters for the resulting importance sampling scheme to be asymptotically optimal.
Moreover, optimality can be achieved by a triangle tilt for every $(p,t)$ in the replica symmetric phase, and even for some choices of $(p,t)$ in the replica breaking phase, as we will show in Section \ref{sec:characterizeRegimes}.
Thus, the triangle tilt succeeds where the edge tilt fails, because the former appropriately penalizes samples with an undesired number of triangles, whereas the latter inappropriately penalizes samples with an undesired number of edges. 
As mentioned in the preceding paragraph, a crucial property to be expected of such an optimal triangle tilt is that samples from the tilted measure resemble the original graph $\calG_{n,p}$ conditioned to have at least ${n \choose 3}t^3$ triangles.
This is indeed the case for the optimal triangle tilt, thanks to Theorem \ref{thm: Exponential graph free energy}. 

Finally, we remark that Theorem \ref{thm: Exponential graph free energy} draws the connection between an exponential random graph and a conditioned \ER graph, indicating how the parameters for the two graphs must be related in order for them to resemble each other. This relationship arises from the fact that the free energy of the exponential random graph can be expressed in a variational formulation involving the LDP rate function for the conditioned \ER graph. For $(p,t)$ in the replica symmetric phase, this connection has been observed by Chatterjee and Dey \cite{ChaDey10}, Chatterjee and Diaconis \cite{ChaDia11}, and Lubetzky and Zhao \cite{LubZhao12}. In this paper, Theorem \ref{thm: Exponential graph free energy} generalizes this connection to include parameters $(p,t)$ in the replica breaking phase.

\vspace{0.2in}

{\bf Organization of the paper:}
 We start by giving precise definitions of the various constructs arising in our study in Section \ref{sec:def}. This culminates in Theorem \ref{thm: Exponential graph free energy} that characterizes the limiting free energy of the exponential random graph model. The rest of Section 2 is devoted to drawing a connection between the exponential random graph and \ER random graph conditioned on an atypical number of triangles, leading to the derivation of the triangle tilts. Section \ref{sec:asympEff} discusses and proves our main results on asymptotic optimality or non-optimality of the importance sampling estimators. An explicit procedure for determining the optimal triangle tilt parameters, given $(p,t)$, is present in Section \ref{sec:characterizeRegimes} and further expanded on in Appendix \ref{sec: (appdx) S_alpha}. In Section \ref{sec:NumSim}, we carry out numerical simulations on moderate size networks using the various proposed tilts to illustrate and compare the viability of the importance sampling schemes. Additionally, we also discuss alternative strategies for choosing the tilt measure, hybrid tilts and conditioned triangle tilts, which are variants of the edge and triangle tilts.

\vspace{0.2in}

\noindent{\bf Acknowledgement} This work was funded in part through the 2011-2012 SAMSI Program on Uncertainty Quantification, in which each of the authors participated.  JN was partially supported by grant NSF-DMS 1007572. SB was partially supported by grant NSF-DMS 1105581.

%%%%%%%%%%%%%%%%%%%%%%%%%%%%%%%%%%%%%%%   
%%%%   %%%%   %%%%   %%%%   %%%%   %%%%   
%%%%%%%%%%%%%%%%%%%%%%%%%%%%%%%%%%%%%%%   
\section{Large deviations, importance sampling and exponential random graphs}
\label{sec:def}

%%%%%%%%%%%%%%%%%%%%%%%%%%%%%%%%%%%%%%%%%%%%%%%%%%%%%%
%%%%%  %%%%%  %%%%%  %%%%%  %%%%%  %%%%%  %%%%%  %%%%%
%%%%%%%%%%%%%%%%%%%%%%%%%%%%%%%%%%%%%%%%%%%%%%%%%%%%%%
\subsection{Large deviations for \ER random graphs}

Before the proof of the main result, we start with a more detailed description of the large deviations principle for \ER random graphs and introduce the necessary constructs required in our proof.
Chatterjee and Varadhan \cite{ChaVar11} have proved a general large deviation principle which is based on the theory of dense graph limits developed by \cite{Borgsetal08} (See also Lovasz's recent monograph, \cite{Lovasz2012monograph}).
In this framework, a random graph is represented as a function $X(x,y) \in \widetilde{\calW}$, where $\widetilde{\calW}$ is the set of all measureable functions $f:[0,1]^2 \to [0,1]$ satisfying $f(x,y) = f(y,x)$. 
Specifically, a finite simple graph $X$ on $n$ vertices is represented by the function, or \emph{graphon},
\begin{equation} \label{eq: graphon}
  X(x,y) = \sum_{\substack{i,j=1 \\i \neq j}}^n X_{ij} \indc_{[\frac{i-1}{n},\frac{i}{n}) \times [\frac{j-1}{n},\frac{j}{n})}(x,y) \,\in \widetilde{\calW}.
\end{equation}
Here we treat $(X_{ij})$ as a symmetric matrix with entries in $\{0,1\}$ and $X_{ii} = 0$ for all $i$.
In general, for a function $f\in \widetilde{\calW}$, $f(x,y)$ can be interpreted as the probability of having an edge between vertices $x$ and $y$.
Then, we define the quotient space $\calW$ under the equivalence relation defined by $f \sim g$ if $f(x,y) = g(\sigma x,\sigma y)$ for some measure preserving bijection $\sigma: [0,1]\rightarrow[0,1]$.
Intuitively, an equivalence class contains graphons that are equal after a relabelling of vertices.
(See, e.g., \cite{Borgsetal08, ChaVar11} for further exploration and properties of the quotient space.)

By identifying a finite graph $X$ with its graphon representation, we can consider the probability measure $\ProbP_{n,p}$ as a measure induced on $\calW$ supported on the finite subset of graphons of finite graphs.
For $f\in\calW$, denote 
\begin{equation}
\label{eqn:edge-f-def}
\calE(f) = \int_0^1 \!\! \int_0^1 f(x,y) \,dx\,dt   
\end{equation}
and 
\begin{equation}
\label{eqn:triangle-f-def}
\calT(f) = \int_0^1 \!\! \int_0^1 \!\! \int_0^1 f(x,y) f(y,z) f(x,z) \,dx\,dy\, dz.     
\end{equation}
We see that $E(X) = \frac{n^2}{2} \calE(X)$ and $T(X) = \frac{n^3}{6} \calT(X)$, so that $\calE$ and $\calT$ represent edge and triangle densities of the graph $X$, respectively. 
Then, rather than considering the event $W_{n,t}$, we shall equivalently consider the upper tails of triangle densities, 
\[
\calW_t := \{ f \in \calW\;|\;\; \calT(f) \geq t^3 \}.
\]

The large deviation principle of Chatterjee and Varadhan \cite{ChaVar11} implies for any $p \in (0,1)$ and $t \in [p,1]$,
\be \label{eq: LDP statement}
\lim_{n \to \infty} \frac{1}{n^2} \log \ProbP \left( \calT(\calG_{n,p}) \geq t^3 \right) = - \phi(p,t)
\ee
where $\phi(p,t)$ is the large deviation decay rate given by a variational form,
\be \label{eq: LDP rate runction}
\phi(p,t) = \inf \left\{ \calI_p(f)\;|\; f \in \mathcal{W}, \;\; \mathcal{T}(f) \geq t^3 \right\}
= \inf_{f\in\calW_t} \,[\calI_p(f)].
\ee
Here,
\be
\mathcal{I}_p(f) := \int_0^1 \! \int_0^1 I_p(f(x,y)) \,dx\,dy
\label{eq: LDP rate function}
\ee
is the large deviation rate function, where $I_p:[0,1] \to \Rm$ is defined at (\ref{Iprate}). A further important consequence of the large deviation principle concerns the typical behaviour of the conditioned probability measure
%\[
%\ProbP_{n,p}(X|\calW_t) = \ProbP_{n,p}(X) {\bf1}_{\calW_t}(X) \mu_n^{-1}.
%\]
\[
\ProbP_{n,p}(A|\calW_t) = \ProbP_{n,p}(A \cap \calW_t) \mu_n^{-1}.
\]
When we refer to $\calG_{n,p}$ conditioned on the event $\calW_t=\set{\calT(f) \geq t^3}$, we mean the random graph whose law is given by this conditioned probability measure.

\begin{lemma} \label{lem: typical conditioned ER graph}
  (\cite[{Theorem 3.1}]{ChaVar11}, Lemma \ref{lem: minimalF^*})
  Let $\calF^* \subset \calW$ be the non-empty set of graphs that optimize the variational form in \eqref{eq: LDP rate runction}. Then the \ER graph $\calG_{n,p}$ conditioned on $\set{\calT(f) \geq t^3}$ is asymptotically indistinguishable from the minimal set $\calF^*$.
\end{lemma}

The term ``asymptotically indistinguishable" in Lemma \ref{lem: typical conditioned ER graph} roughly means that the graphon representation of the graph converges in probability, under the \emph{cut distance} metric, to some function $f^*\in \calF^*$ at an exponential rate as $n \to \infty$. 
Intuitively, this means that the typical conditioned \ER graph resembles some graph $f^\ast \in \calF^*$ for large $n$.
In order to give a more precise definition of asymptotic indistinguishability, we first recall the cut distance metric $\delta_{\square}$, defined for $f,g \in \calW$ by
\[
\delta_\square(f,g) = \inf_{\sigma} \sup_{S,T \subset [0,1]} \left| \int_{S \times T} (f(\sigma x,\sigma y) - g(x,y)) \,dx \,dy\right|,
\] 
where the infimum is taken over all measure-preserving bijections $\sigma:[0,1] \to [0,1]$.
For $\calF_1, \calF_2 \subset \calW$,
\[
\delta_\square(\calF_1,\calF_2) = \inf_{f_1 \in \calF_1, f_2\in \calF_2} \delta_\square(f_1,f_2).
\] 
It is known by \cite{LovSze06} that $(\calW,\delta_\square)$ is a compact metric space. 

We say that a family of random graphs $\calG_{n}$ on $n$ vertices, for $n \in \mathbb{N}$, is asymptotically indistinguishable from a subset $\calF \subset \calW$ if: for any $\epsilon_1 > 0$ there is $\epsilon_2> 0$ such that
\begin{equation} \label{eq: asymp indistinguishable}
\limsup_{n \to \infty} \frac{1}{n^2} \log \ProbP( \delta_{\square}(\calG_{n}, \calF) > \epsilon_1) < - \epsilon_2  .
\end{equation}
Further, we say that $\calG_{n}$ is asymptotically indistinguishable from the \emph{minimal} set $\calF \subset \calW$ if $\calF$ is the smallest closed subset of $\calW$ that $\calG_{n}$ is asymptotically indistinguishable from.
Clearly, if $\calG_{n}$ is asymptotically indistinguishable from a singleton set $\calF$, then $\calF$ is, trivially, minimal.
Finally, we say two random graphs $\calG_n^1$, $\calG_n^2$ are asymptotically indistinguishable if they are each asymptotically indistinguishable from the same minimal set $\calF\subset\calW$. 
Intuitively, this means that the random behaviour, or the typical graphs, of $\calG_n^1$ resembles that of $\calG_n^2$ for large $n$.
(See \cite{ChaDia11} and \cite{ChaVar11} for a wide-ranging exploration of this metric in the context of describing limits of dense random graph sequences.)

Using this terminology, we observe that an \ER graph $\calG_{n,u}$ is asymptotically indistinguishable from the singleton set containing the constant function $f^* \equiv u$.
A special notion about whether the conditioned \ER graph is again an \ER graph leads to the following definition.

\begin{definition} \label{def:replica symmetric}
  The \emph{replica symmetric phase} is the regime of parameters $(p,t)$ for which the large deviations rate satisfies
\be
\inf_{f\in\calW_t} \left[ \calI_p(f)\right] = I_p(t), \label{Ipmin}
\ee
and the infimum is uniquely attained at the constant function $t$.

The \emph{replica breaking phase} is the regime of parameters $(p,t)$ that are not in the replica symmetric phase.
\qed
\end{definition}

Hence, the notion of replica symmetry is a property of the rare event problem, where the \ER graph $\calG_{n,p}$ conditioned on the event $\{\calT(f) \geq t^3\}$ is asymptotically indistinguishable from an \ER graph with the higher edge density, $\calG_{n,t}$, a consequence of Lemma \ref{lem: typical conditioned ER graph}.
In contrast, the conditioned graphs in the replica breaking phase are not indistinguishable from any one \ER graph; instead, they may behave like a mixture of \ER graphs or exhibit a clique-like structure with edge density less than $t$.
The term ``replica symmetric phase'' is borrowed from \cite{ChaVar11}, which in turn was inspired by the statistical physics literature.
However, we remark that this term has been used differently from us by other authors to refer to other families of graphs behaving like an \ER graph or a mixture of \ER graphs.

%%%%%%%%%%%%%%%%%%%%%%%%%%%%%%%%%%%%%%%%%%%%%%%
\subsection{Asymptotic behavior of exponential random graphs}
\label{sec: Asymptotic behavior of exponential graphs}

To find ``good'' importance sampling tilted measures, we focus on the class of exponential random graphs.
The exponential random graph is a random graph on $n$ vertices defined by the Gibbs measure
\begin{equation}
  \ProbQ (X) = \ProbQ^{h,\beta,\alpha}_n (X) \propto e^{n^2 \calH(X)} \label{eq: Gibbs measure}
\end{equation}
on $\Omega_n$, where for given $h \in \Rm$, $\beta\in \Rm_{+}, \alpha >0$, the Hamiltonian is
\begin{equation}
\label{eqn:hamiltonian-def}
    \calH(X) = \frac{h}{2} \calE(X) +\frac{\beta}{6} \calT(X)^{\alpha}.
\end{equation}
We will use $\psi_n = \psi_n^{h,\beta,\alpha} $ to denote the log of the  normalizing constant (free energy)
\[
\psi_n = \psi_n^{h,\beta,\alpha} = \frac{1}{n^2} \log \sum_{X \in \Omega_n} e^{n^2 \calH(X)},
\]
so that $\ProbQ^{h,\beta,\alpha}_n(X) = \exp(n^2(\calH(X) - \psi_n))$.  We denote by $\calG_{n}^{h,\beta,\alpha}$ the exponential random graph defined by the Gibbs measure \eqref{eq: Gibbs measure}. 
The case where $\alpha=1$ is the ``classical'' exponential random graph model that has an enormous literature in the social sciences, see e.g. \cite{robins2007recent,robins2007introduction} and the references therin and rigorously studied in a number of recent papers, see e.g. \cite{BhaBreSly08,ChaDia11,RadYin12,LubZhao12,yin2012cluster,yin2012critical}. In this case, the Hamiltonian can be rewritten as $n^2 \calH(X) = h E(X) + \frac{\beta}{n} T(X)$. We will drop the superscripts in $\psi_n^{h,\beta}, \ProbQ^{h,\beta}_n$ when $\alpha=1$.
The generalization to the exponential random graph with the parameter $\alpha$ was first proposed in \cite{LubZhao12}.

Observe that the \ER random graph is a special case of the exponential random graph: if $\beta=0$ and $h=h_p$ with $h_p$ defined by \eqref{eq:h_p}, then $\ProbQ^{h_p,0,\alpha}_n = \ProbP_{n,p}$ for any $\alpha>0$ and the edges are independent with probability $p$.
On the other hand, choosing $\beta > 0$ introduces a non-trivial dependence between the edges.
By adjusting the parameters $(h,\beta,\alpha)$, the Gibbs measure $\ProbQ_{n}^{h,\beta,\alpha}$ can be adjusted to favor edges and triangles to varying degree.

The asymptotic behavior of the exponential random graph measures $\ProbQ^{h,\beta,\alpha}_n$ and the free energy $\psi_n^{h,\beta,\alpha}$ is partially characterized by the following result of Chatterjee and Diaconis \cite{ChaDia11} and Lubetzky and Zhao \cite{LubZhao12}. 
In what follows, we will make use of the functions
\be \label{eq: I(u)}
I(u) = \frac{1}{2} u\log u + \frac{1}{2} (1-u) \log (1-u)
\ee
on $u\in[0,1]$ and, for $f\in\calW$,
\be \label{eq: calI(f)}
\mathcal{I}(f) := \int_0^1\!\int_0^1 I(f(x,y)) \,dx\,dy.
\ee

\begin{theorem}[See \cite{ChaDia11} \cite{LubZhao12}] \label{thm: ChaDai11 Thm 4.1}
For the exponential random graph $\calG_{n}^{h,\beta,\alpha}$ with parameters $(h,\beta,\alpha)\in \bR\times \bR^+\times [2/3,1]$, the free energy satisfies
  \begin{equation} \label{FEvar alpha}
    \lim_{n \to \infty} \psi_n^{h,\beta, \alpha} = \sup_{0\leq u\leq 1} \bigg[\frac{\beta}{6} u^{3\alpha} - I(u) + \frac{h}{2} u  \bigg].
  \end{equation}
If the supremum in \eqref{FEvar alpha} is attained at a unique point $v^* \in [0,1]$, then the exponential random graph $\calG_{n}^{h,\beta,\alpha}$ is asymptotically indistinguishable from the \ER graph $\calG_{n,v^*}$.
\end{theorem}

The case $\alpha = 1$ in Theorem \ref{thm: ChaDai11 Thm 4.1} was proved by Chatterjee and Diaconis -- Theorems 4.1, 4.2 of \cite{ChaDia11}; the cases $\alpha \in [2/3,1]$ is due to Lubetzky and Zhao-- Theorems 1.3, 4.3 of \cite{LubZhao12}.

Our main result in this section, stated next, is the generalization of the variational formulation for the free energy of the Gibbs measure of any exponential random graph.
Our result emphasizes the connection between the exponential random graph and the conditioned \ER graph. Before stating the result we will need some extra notation. Extend the Hamiltonian defined in \eqref{eqn:hamiltonian-def} to the space of graphons in the natural way 
\begin{equation}
\label{eqn:graphon-hamiltonian}
    \cH(f):= \frac{h}{2} \cE(f) + \frac{\beta}{6} \cT(f)^\alpha
\end{equation}
where recall the definitions for the density of edges and triangles for graphons defined respectively in \eqref{eqn:edge-f-def} and \eqref{eqn:triangle-f-def}. For fixed $q\in (0,1)$ recall the functions $\cI_q(f)$ from \eqref{eq: LDP rate function} and the function $\cI(f)$ from \eqref{eq: calI(f)}. In particular, observe that
\be
\mathcal{I}_q(f) = \mathcal{I}(f) - \frac{h_q}{2} \mathcal{E}(f) - \frac{1}{2} \log(1 - q). \label{IpIqident}
\ee
with $h_q = \log \frac{q}{1-q}$.

\begin{theorem} \label{thm: Exponential graph free energy}
For the exponential random graph $\calG_{n}^{h,\beta,\alpha}$ with parameters $(h,\beta,\alpha)\in \bR\times [0,+\infty) \times [0,1]$, the free energy satisfies
  \begin{align} \label{eq: free energy}
    \lim_{n\rightarrow\infty} \psi_{n}^{h,\beta,\alpha}  &= \sup_{0\leq u\leq1} \left[\frac{\beta}{6}u^{3\alpha} - \phi_q(u) - \frac{1}{2}\log(1-q) \right]
  \end{align}
where $q \in (0,1)$ is such that $h=h_q=\log\frac{q}{1-q}$, and
\begin{equation} \label{eq: phi_q(u)}
  \phi_q(u) = \inf_{f \in \partial \calW_{u}} [\calI_q(f)]
\end{equation}
and $\partial \calW_u := \{ f\in\calW \,|\, \calT(f) = u^3\}$.

If the supremum in \eqref{eq: free energy} is attained at a unique point $v^* \geq q$, then the exponential random graph $\calG_{n}^{h_q,\beta,\alpha}$ is asymptotically indistinguishable from the conditioned \ER graph, $\calG_{n,q}$ conditioned on the event $\set{ \calT(f) \geq (v^*)^3 }$.
\end{theorem}

\begin{remark}
\begin{enumerate}[(i)]
  \item If, in addition, $(q,v^*)$ in Theorem \ref{thm: Exponential graph free energy} belongs to the replica symmetric phase, then $\calG_{n}^{h_q,\beta,\alpha}$ is asymptotically indistinguishable from the \ER graph $\calG_{n,v^*}$. This follows from the remarks following Definition \ref{def:replica symmetric}, that in the replica symmetric phase, $\calG_{n,q}$ conditioned on $\set{ \calT(f) \geq (v^*)^3 }$ is asymptotically indistinguishable from the \ER graph $\calG_{n,v^*}$. 
In this case, \eqref{eq: free energy} reduces to \eqref{FEvar alpha}.
%Thus, Theorem \ref{thm:  ChaDai11 Thm 4.1} is a special case of Theorem \ref{thm: Exponential graph free energy}.

  \item Non-uniqueness of $v^*$ is possible. As will be apparent from the proof, if the supremum in \eqref{eq: free energy} is attained on the set $U^* \subset [0,1]$, then the exponential random graph $\calG_n^{h_q,\beta,\alpha}$ is asymptotically indistinguishable from the minimal set $\calF^* = \bigcup_{u\in U^*} \calF^*_u$, where $\calF^*_u$ is the set of minimizers of \eqref{eq: phi_q(u)}. In particular, if $U^*$ contains more than one element, then $\calG_n^{h,\beta,\alpha}$ is asymptotically indistinguishable from a mixture of different conditioned \ER graphs.

\end{enumerate}

\end{remark}

\begin{proof}
Theorem 3.1 in \cite{ChaDia11} implies that
\begin{align} \label{eq: free energy var form}
\lim_{n\rightarrow\infty} \psi_{n}^{h_q,\beta,\alpha} = \sup_{f\in \calW} [\calH(f)-\calI(f)].
\end{align}
To show (\ref{eq: free energy}), suppose $f \in \partial \calW_u$, for $u\in (0,1)$. Recalling (\ref{IpIqident}), we have
\begin{align} \label{eq: H-I bound}
  \calH(f) - \calI(f)
  &= \frac{h_q}{2} \calE(f) + \frac{\beta}{6} u^{3\alpha} - \calI(f) \nonumber\\
  &= \frac{\beta}{6} u^{3\alpha} - \calI_q(f) - \frac{1}{2}\log(1-q) \\
  &\leq \frac{\beta}{6} u^{3\alpha} - \inf_{f\in\partial\calW_u} [\calI_q(f)] - \frac{1}{2}\log(1-q) \nonumber
\end{align}
This implies that
\begin{align*}
  \sup_{f\in \partial \calW_u} [\calH(f) - \calI(f)]
  &\leq \frac{\beta}{6} u^{3\alpha} - \inf_{f\in \partial \calW_u} [\calI_q(f)] - \frac{1}{2}\log(1-q),
\end{align*}
and  
\begin{align*}
  \sup_{f \in \calW} [\calH(f) - \calI(f)] 
  &= \sup_{0 \leq u \leq 1} \, \sup_{f \in \partial \calW_u} [\calH(f) - \calI(f)] \\
  &\leq \sup_{0\leq u\leq1} \left[ \frac{\beta}{6} u^{3\alpha} - \inf_{f\in \partial \calW_u} [\calI_q(f)] - \frac{1}{2}\log(1-q) \right]
\end{align*}
Now we show the reverse inequality. Fix $\epsilon > 0$. For each $u\in(0,1)$, let $f_{u,\epsilon} \in \partial \calW_u$ be such that 
\[
\calI_q(f_{u,\epsilon}) \leq \inf_{f\in \partial \calW_u} [\calI_q(f)] +\epsilon.
\]
Therefore, for each $u \in (0,1)$ we have
\br
\mathcal{H}(f_{u,\epsilon}) - \mathcal{I}(f_{u,\epsilon}) & = &  \frac{\beta}{6} u^{3\alpha} - [\calI_q(f_{u,\epsilon})] - \frac{1}{2}\log(1-q) \no \\
& \geq &  \frac{\beta}{6} u^{3\alpha} - \inf_{f\in \partial \calW_u} [\calI_q(f)] - \frac{1}{2}\log(1-q) - \epsilon .
\er
Hence
\br
  \sup_{f \in \calW} [\calH(f) - \calI(f)] \geq \sup_{0 < u < 1} \left[ \frac{\beta}{6} u^{3\alpha} - \inf_{f\in \partial \calW_u} [\calI_q(f)] - \frac{1}{2}\log(1-q) \right] - \epsilon.
\er
Since $\epsilon > 0$ is arbitrary, \eqref{eq: free energy} follows.

To show the next statement, suppose the supremum in \eqref{eq: free energy} is attained at a unique point $v^* \geq q$. 
Let $\calF^*_{v^*}\subset \partial \calW_{v^*}$ denote the set of functions that attains the infimum in \eqref{eq: phi_q(u)}.
We observe from the preceeding proof that, in fact, $\calF^*_{v^*}$ is the set that attains the infimum in \eqref{eq: free energy var form}, so that by \cite[{Theorem 3.2}]{ChaDia11} and Lemma \ref{lem: minimalF^*}, the graph $\calG_{n}^{h_q,\beta,\alpha}$ is asymptotically indistinguishable from the minimal set $\calF^*_{v^*}$.
On the other hand, since $\calF^*_{v^*}$ is also the set that attains the infimum in the LDP rate in  \eqref{eq: LDP rate runction} (due to \cite[{Theorem 4.2(iii)}]{ChaVar11}), the conditioned \ER graph, $\calG_{n,q}$ conditioned on $\set{ \calT(f) \geq (v^*)^3 }$, is also asymptotically indistinguishable from the set $\calF_{v^*}^*$.
Thus, $\calG_{n}^{h_q,\beta,\alpha}$ is asymptotically indistinguishable from the conditioned \ER graph, $\calG_{n,q}$ conditioned on the event $\set{ \calT(f) \geq (v^*)^3 }$.

\end{proof}

The mean behaviour of the triangle density of an exponential random graph $\calG_n^{h_q,\beta,\alpha}$ can be deduced from the variational formulation in \eqref{eq: free energy}, and in special instances, so can the mean behaviour of the edge density.
This is shown in the next proposition, which follows from \cite[{Theorem 4.2}]{ChaDia11} and the Lipschitz continuity of the mappings $f \mapsto \calT(f)$ and $f \mapsto \calE(f)$ under the cut distance metric $\delta_{\square}$ \cite[{Theorem 3.7}]{Borgsetal08}. The proof is left to the appendix.

\begin{proposition} \label{prop: expo graph mean behaviour}
Let $(h_q,\beta,\alpha)\in \bR\times [0,+\infty) \times [0,1]$. If the supremum in \eqref{eq: free energy} is attained at a unique point $v^* \in [0,1]$, then
\be
\lim_{n\rightarrow \infty} \Mean | \calT(\calG_n^{h_q,\beta,\alpha}) - (v^*)^3| = 0. \label{QTmean2}
\ee
Further, if $(q,v^*)$ belongs to the replica symmetric phase, then
\be
\lim_{n\rightarrow \infty} \Mean | \calE(\calG_n^{h_q,\beta,\alpha}) - v^*| = 0. \label{QEmean2}
\ee
\end{proposition}

%%%%%%%%%%%%%%%%%%%%%%%%%%%%%%%%%%%
%%%%%%%   NEW SECTION    %%%%%%%%%%
%%%%%%%%%%%%%%%%%%%%%%%%%%%%%%%%%%%
\section{Asymptotic Optimality}
\label{sec:asympEff}

Recall that the edge tilt corresponds to the Gibbs measure (\ref{expGibbsdef}) with $\beta = 0$ and $h > h_p = \log \frac{p}{1-p}$. Thus, an edge tilt $\ProbQ_{n}^{h,0}$ satisfies
\br
\frac{d\ProbP_{n,p}}{d\ProbQ_{n}^{h,0}} (X) =  \exp \left[- n^2 \left(\frac{h-h_p}{2} \calE(X) + \psi_n^{h_p,0} - \psi_n^{h,0} \right) \right].
\er
The triangle tilt corresponds to the Gibbs measure (\ref{expGibbsdef}) with $h = h_p$ and $\beta > 0$, $\alpha > 0$. So, the triangle tilt $\ProbQ_{n}^{h_p,\beta,\alpha}$ satisfies
\[
\frac{d\ProbP_{n,p}}{d\ProbQ_{n}^{h_p,\beta,\alpha}} (X) =  \exp \left[- n^2 \left(\frac{\beta}{6} \calT(X)^{\alpha} + \psi_n^{h_p,0} - \psi_n^{h_p,\beta,\alpha}\right) \right]
\]
Here recall that $\cT(X)= \frac{6}{n^3}T(X)$ is the density of triangles in $X$ and $\cE(X)= \frac{2}{n^2} E(X)$ is the density of edges.

For any admissible parameters $(h,\beta,\alpha)$, the importance sampling estimator based on the tilted measure $\ProbQ_{n}^{h,\beta,\alpha}$ is
\begin{align} \label{eq: IS estimator}
\tilde M_K &= \frac{1}{K} \sum_{k=1}^K {\bf1}_{\calW_t}(\tilde{X}_k) \frac{d\ProbP_{n,p}}{d\ProbQ_{n}^{h,\beta,\alpha}} (\tilde{X}_k) \nonumber\\
  &= \frac{1}{K} \sum_{k=1}^K {\bf1}_{\calW_t}(\tilde{X}_k) 
  \exp \left\{n^2 \left( \frac{h_p-h}{2} \calE(\tilde{X}_k) - \frac{\beta}{6} \calT(\tilde{X}_k)^{\alpha} + \psi_n^{h,\beta,\alpha} - \psi_n^{h_p,0} \right) \right\} 
\end{align}
where $\tilde{X}_k$ are i.i.d. samples drawn from $\ProbQ_{n}^{h,\beta,\alpha}$. Denote 
\begin{equation*}
  \hat{q}_n = \hat{q}_{n} (\tilde{X})
  = {\bf1}_{\calW_t}(\tilde{X}) \frac{d\ProbP_{n,p}}{d\ProbQ_{n}^{h,\beta,\alpha}} (\tilde{X}).
\end{equation*}
For any $(h,\beta,\alpha)$, $\Mean [\hat{q}_n] = \mu_{n}$ and so $\tilde M_K$ is an unbiased estimator for $\mu_{n}$.

Our first result is a necessary condition for asymptotic optimality of the importance sampling scheme:

\begin{proposition} \label{prop: ERu*NoOpt}
 Given $p < t$, let $(h,\beta,\alpha)\in \mathbb{R} \times [0,+\infty) \times [0,1]$ with $h = h_q = \log \frac{q}{1-q}$. Suppose that the supremum in \eqref{eq: free energy} is not attained at $t$:
\begin{align}  \label{eq: necessary cond for opt}
\sup_{0\leq u\leq1} \left[\frac{\beta}{6}u^{3\alpha} - \phi_q(u) - \frac{1}{2}\log(1-q) \right] \neq \frac{\beta}{6}t^{3\alpha} - \phi_q(t) - \frac{1}{2}\log(1-q).
\end{align}
  Then the importance sampling scheme based on the Gibbs measure tilt $\ProbQ_{n}^{h,\beta,\alpha}$ is not asymptotically optimal.
\end{proposition}

\begin{corollary}
 Given $p < t$, let $(h,\beta,\alpha)\in \mathbb{R} \times [0,+\infty) \times [0,1]$ with $h = h_q = \log \frac{q}{1-q}$. Suppose that family of random graphs $\calG_{n}^{h,\beta,\alpha}$ is not indistinguishable from $\calG_{n,p}$ conditioned on the event $\{\mathcal{T}(X) \geq t^3\}$.    Then the importance sampling scheme based on the Gibbs measure tilt $\ProbQ_{n}^{h,\beta,\alpha}$ is not asymptotically optimal.
\end{corollary}

Our next result shows that for triangle tilts (i.e. $h = h_p$), the necessary condition described in Proposition \ref{prop: ERu*NoOpt} is also a sufficient condition for asymptotic optimality:

\begin{proposition} \label{prop: asymptotic optimality} Given $p < t$, let $h = h_p$ and $(\beta,\alpha)\in [0,+\infty) \times [0,1]$. Suppose that the supremum in \eqref{eq: free energy} is attained at $t$:
\begin{align} \label{trianglesufficient}
\sup_{0\leq u\leq1} \left[\frac{\beta}{6}u^{3\alpha} - \phi_p(u) - \frac{1}{2}\log(1-p) \right] = \frac{\beta}{6}t^{3\alpha} - \phi_p(t) - \frac{1}{2}\log(1-p).
\end{align}
Then the importance sampling scheme based on the triangle tilt $\ProbQ_{n}^{h_p,\beta,\alpha}$ is asymptotically optimal.
\end{proposition}

In Section \ref{sec:characterizeRegimes}, we give a more explicit way to determine the tilt parameters that satisfy the condition \eqref{trianglesufficient}.

Next, we turn to the edge tilts (i.e. $\beta=0$).
Since $\phi_q(u)$ is minimized at $u=q$ and \eqref{eq: necessary cond for opt} holds if $\beta=0$ and $q\neq t$, we have, as a corollary of Prop \ref{prop: ERu*NoOpt}, the following necessary condition for an edge tilt to produce an optimal scheme.

\begin{proposition} \label{prop:edgeNo}
Given $p < t$, let $\beta = 0$ and $h = h_q$ for some $q \neq t$. The importance sampling scheme based on the edge tilt $\ProbQ_{n}^{h_q,0,\alpha}$ is not asymptotically optimal.
\end{proposition}

Observe that for the edge tilt with $h = h_t$ and $\beta = 0$, the supremum in \eqref{eq: free energy} is always attained at $t$, since
\br
\inf_{0\leq u\leq1} \left[ \phi_t(u) \right] = \inf_{f \in \calW} \left[ \mathcal{I}_t(f) \right] =  \phi_t(t) = 0.
\er
Thus, the edge tilt with $h = h_t$ always satisfies the necessary condition for asymptotic optimality of the importance sampling scheme. Furthermore, if $(p,t)$ is in the replica symmetric phase, the tilted measure $\ProbQ_{n}^{h_t,0,\alpha}$ is indistinguishable from the conditioned \ER graph. Nevertheless, the sampling scheme based on the edge tilt with $h = h_t$ may still be suboptimal, even in the replica symmetric phase, as the next result shows.

%%%%%%%%%%%%%%%%%%%%%%%%%%%%%%
\begin{proposition} \label{prop:edgeNoOpt} 
Let $0 < p < \frac{e^{-1/2}}{1+e^{-1/2}}$ and $t \in (p,1)$. If $t$ is sufficiently close to $1$ and $(p,t)$ belong to the replica symmetric phase, then the importance sampling scheme based on the edge tilt $\ProbQ^{h_t,0}_n$ is not asymptotically optimal.
\end{proposition}

Remark \ref{remark:param_edgeNoOpt} and Figure \ref{fig:replicaphase} indicate that there do exist parameters $(p,t)$ belonging to the replica symmetric phase for which the hypothesis of Prop \ref{prop:edgeNoOpt} is satisfied.

\subsection{Proofs of results.}

We first prove the asymptotic optimality of the triangle tilts, Prop \ref{prop: asymptotic optimality}.
%%%%%%%%%%%%%%%%%%%%%%%%%%%%%%

\noindent

\begin{proof}[Proof of Proposition \ref{prop: asymptotic optimality}]
Due to \eqref{eq: Jensen's asymptotic lower bound}, it suffices to show that
\be
\lim_{n\rightarrow \infty} \frac{1}{n^2} \log \Mean_{\ProbQ} [\hat{q}_n^2] \leq -2 \inf_{f \in \calW_t} \mathcal{I}_p(f). \label{opttiltbound}
\ee
Let $q \in (0,1)$ be such that $h = h_q = \log \frac{q}{1 - q}$. Recall that
\[
\ProbQ_{n}^{h,\beta,\alpha}(X) = \exp\left[ n^2 \left(\frac{h}{2}\mathcal{E}(X) + \frac{\beta}{6} \mathcal{T}^\alpha(X) - \psi_n^{h,\beta,\alpha} \right)\right].
\]
Therefore, by definition of $\hat q_n$, we have
\begin{eqnarray}
\Mean_{\ProbQ_n} [\hat{q}_n^2] & = & \Mean_{\ProbP_{n,p}} \left[ {\bf1}_{\calW_t} \, \frac{d\ProbP_{n,p}}{d\ProbQ_{n}^{h_q,\beta,\alpha}} \right]  \no \\
& = & \Mean_{\ProbP_{n,p}} 
\left[ \exp \left\{ n^2 \left( {\bf0}_{\calW_t}(X) + \frac{h_p-h_q}{2}\calE(X) - \frac{\beta}{6}\calT(X)^\alpha  + \psi_n^{h_q,\beta,\alpha} - \psi_n^{h_p,0} \right) \right\} \right], \no
\end{eqnarray}
where ${\bf1}_{\calW_t}(X) = e^{n^2 {\bf0}_{\calW_t}(X)} $ with ${\bf0}_{\calW_t}(X) =0$ if $X\in\calW_t$ and ${\bf0}_{\calW_t}(X) =-\infty$ otherwise. The mappings $\calE, \calT : \calW \mapsto \Rm $ are bounded and continuous  \cite[Theorem 3.8]{Borgsetal08}, and the function ${\bf0}_{\calW_t}(X)$ can be approximated by bounded continuous approximations. Applying the Laplace principle for the family of measures $\ProbP_{n,p}$, for which $\calI_p(f)$ is the rate function \cite[Theorem 3.1]{ChaDia11}, we obtain

\br
\lim_{n\rightarrow \infty} \frac{1}{n^2} \log \Mean_{\ProbQ_n} [\hat{q}_n^2] & = &
 \lim_{n\rightarrow \infty} \frac{1}{n^2} \log \Mean_{\ProbP_{n,p}} \left[ \exp \left\{ n^2 \left( {\bf0}_{\calW_t}(X) + \frac{h_p-h_q}{2}\calE(X) - \frac{\beta}{6}\calT(X)^\alpha  \right) \right\} \right] \no \\
& & + \lim_{n\rightarrow \infty} \left( \psi_n^{h_q,\beta,\alpha} - \psi_n^{h_p,0} \right) \no \\
& = & - \inf_{f\in\calW_t} \left[ \calI_p(f) + \frac{h_q-h_p}{2} \calE(f) + \frac{\beta}{6} \calT(f)^{\alpha} \right] 
+ \lim_{n\rightarrow \infty} \left( \psi_n^{h_q,\beta,\alpha} - \psi_n^{h_p,0}\right) \label{qnlim1} 
\er
By \eqref{eq: free energy},
\[
\lim_{n\rightarrow \infty} \psi_n^{h_q,\beta,\alpha} 
= V(u^*)
\]
where $u^\ast = \argsup{0\leq u\leq1} [V(u)]$ and
\[
V(u) := \frac{\beta}{6} u^{3\alpha} - \inf_{f\in\partial \calW_u} [\calI_q(f)] - \frac{1}{2} \log(1-q) .
\]
Also, $\lim_{n\rightarrow \infty} \psi_n^{h_p,0} = -\frac{1}{2} \log(1-p)$. Hence,
\br
&&\lim_{n\rightarrow \infty} \frac{1}{n^2} \log \Mean_{\ProbQ} [\hat{q}_n^2] \label{qnlim2} \\
&& \quad   
= - \inf_{f\in\calW_t} \left[ \mathcal{I}_p(f) + 
   \frac{h_q - h_p}{2} \calE(f) + \frac{\beta}{6} \calT(f)^\alpha \right] + \frac{\beta}{6} (u^*)^{3\alpha} - \inf_{f\in\partial \calW_{u^*}} [\mathcal{I}_q(f)] - \frac{1}{2} \log\frac{1-q}{1-p} \no\\
&& \quad  
\leq - \inf_{f\in\calW_t} \left[ \mathcal{I}_p(f) + \frac{h_q - h_p}{2} \calE(f) \right] + \frac{\beta}{6} \left((u^*)^{3\alpha}-t^{3\alpha}\right) -\inf_{f\in\partial \calW_{u^*}} [\mathcal{I}_q(f)] - \frac{1}{2} \log\frac{1-q}{1-p}  \no
\er
The last inequality follows from the fact that $\calT(f) \geq t^3$ for all $f \in \calW_t$. Since,
\[
\mathcal{I}_q(f) + \frac{1}{2} \log \frac{1 - q}{1 -p} = \mathcal{I}_p(f)  + \frac{h_p - h_q}{2}\mathcal{E}(f), 
\]
we conclude that
\br
\lim_{n\rightarrow \infty} \frac{1}{n^2} \log \Mean_{\ProbQ} [\hat{q}_n^2] & \leq &- \inf_{f\in\calW_t} \left[ \mathcal{I}_p(f) + \frac{h_q - h_p}{2} \calE(f) \right]  -\inf_{f\in\partial \calW_{t}} \left[\mathcal{I}_p(f) -  \frac{h_q - h_p}{2} \calE(f)\right]  \no \\
& & + \frac{\beta}{6} \left((u^*)^{3\alpha}-t^{3\alpha}\right) \label{qnlim3}
\er
The estimate \eqref{qnlim3} holds for any $(h_q,\beta,\alpha) \in \Rm \times[0,+\infty) \times [0,1]$. However, under the hypotheses of Proposition \ref{prop: asymptotic optimality}, we have $u^* = t$ and $q = p$. Therefore,
\br
\lim_{n\rightarrow \infty} \frac{1}{n^2} \log \Mean_{\ProbQ} [\hat{q}_n^2] \leq - \inf_{f\in\calW_t} \left[ \mathcal{I}_p(f)  \right]  -\inf_{f\in\partial \calW_{t}} \left[\mathcal{I}_p(f) \right]   = -2 \inf_{f\in\calW_t} \left[ \calI_p(f) \right]\no
\er
Combined with the upper bound for the asymptotic second moment, we conclude that the triangle tilt $\ProbQ_{n}^{h_p,\beta,\alpha}$ yields an asymptotically optimal importance sampling estimator if (\ref{trianglesufficient}) holds.
\end{proof}

We now prove the necessary condition for optimality, Prop \ref{prop: ERu*NoOpt}.

\begin{proof}[Proof of Proposition \ref{prop: ERu*NoOpt}]
We recall from \eqref{eq: free energy var form} that $\lim_{n\rightarrow\infty} \psi^{h_q,\beta,\alpha}_n = \sup_{f\in\calW} [\calH(f)-\calI(f)]$.
Due to Theorem \ref{thm: Exponential graph free energy}, there exists $f^*\in \calW$ such that $f^*$ minimizes the LDP rate function $\inf_{f\in\calW_t} [\calI_p(f)]$, and $f^*$ does not maximize $\sup_{f\in\calW} [\calH(f) - \calI(f)]$.
From \eqref{qnlim1},
\begin{align*}
&\lim_{n\rightarrow \infty} \frac{1}{n^2} \log \Mean_{\ProbQ_n} [\hat{q}_n^2] \\
&\qquad= - \inf_{f\in\calW_t} \left[ \calI_p(f) + \frac{h-h_p}{2} \calE(f) + \frac{\beta}{6} \calT(f)^{\alpha} \right] 
+ \lim_{n\rightarrow \infty} \psi_n^{h,\beta,\alpha} + \frac{1}{2} \log(1-p) \\
&\qquad= - \inf_{f\in\calW_t} \left[ \calI_p(f) + \frac{h-h_p}{2} \calE(f) + \frac{\beta}{6} \calT(f)^{\alpha} \right] 
+ \sup_{f\in\calW} [\calH(f) - \calI(f)] + \frac{1}{2} \log(1-p) \\
&\qquad > -\left[ \calI_p(f^*) + \frac{h-h_p}{2} \calE(f^*) + \frac{\beta}{6} \calT(f^*)^{\alpha} \right] + \frac{h}{2} \calE(f^*) + \frac{\beta}{6} \calT(f^*)^{\alpha} - \calI(f^*) + \frac{1}{2} \log(1-p)  \\
&\qquad= -2 \calT(f^*) = -2\inf_{f\in\calW_t} [\calI_p(f)]
\end{align*}
Hence the importance sampling estimator is not asymptotically optimal.
\end{proof}

\begin{proof}[Proof of Proposition \ref{prop:edgeNoOpt}] 
For the edge tilt with $h=h_t,\beta=0$, we have from \eqref{qnlim2},
\begin{equation}  \label{qn2lim1}
  \lim_{n\rightarrow \infty} \frac{1}{n^2} \log \Mean_{\ProbQ} [\hat{q}_n^2]
  = - \inf_{f\in\calW_t} \left[ \mathcal{I}_p(f)  +
   \left(\frac{h_t - h_p}{2}\right) \calE(f) \right]
   -  I_p(t)  + \left(\frac{h_t - h_p}{2} \right)t .
\end{equation}
Because $(p,t)$ is in the replica symmetric phase, the term $\mathcal{I}_p(f)$ is minimized over $\calW_t$ by the constant function,
\begin{equation*}
  f_t(x,y) \equiv t = \arginf{f\in\calW_t} [\mathcal{I}_p(f)].
\end{equation*}
On the other hand, the term $\calE(f)$ is minimized over $\calW_t$ by the clique function 
\begin{equation} \label{eq: clique function}
  g_t(x,y) = {\bf1}_{[0,t]^2}(x,y) = \arginf{f\in\calW_t} [\calE(f)].
\end{equation}
This $g_t$ represents a graph with a large clique, in which there is a complete subgraph on a fraction $t$ of the vertices. 
Let $\mathcal{V}(f) = \calI_p(f) + \frac{h_t-h_p}{2} \calE(f)$ be the function to be infimized in \eqref{qn2lim1}. We have 
\br
\mathcal{V}(f_t) & = & I_p(t) + \left(\frac{h_t - h_p}{2}\right)t, \no
\er
and
\br
\mathcal{V}(g_t) & = & t^2 I_p(1) + (1 - t^2) I_p(0) + \left(\frac{h_t - h_p}{2}\right)t^2. \no
\er
Thus, if we can show that 
\begin{align*}
  \mathcal{V}(g_t) < \mathcal{V}(f_t),
\end{align*}
it will follow that from \eqref{qn2lim1} that
\begin{align}
\lim_{n\rightarrow \infty} \frac{1}{n^2} \log \Mean_{\ProbQ} [\hat{q}_n^2] 
  &\geq - \mathcal{V}(g_t) - I_p(t) + \left(\frac{h_t-h_p}{2}\right) t  \nonumber \\
  &> - 2I_p(t) = -2 \inf_{f \in \calW_t} \mathcal{I}_p(f). \label{qn2lim3}
\end{align}

We claim that for $p < \frac{e^{-1/2}}{1+e^{-1/2}}$ and $t$ sufficiently close to $1$, we have $\mathcal{V}(g_t) < \mathcal{V}(f_t)$. Indeed, let
\begin{align} \label{eq: G(t)}
G(t) &:= \mathcal{V}(g_t) - \mathcal{V}(f_t) \\
& = t^2 I_p(1) + (1 - t^2)I_p(0) - I_p(t) +  \left(\frac{h_t - h_p}{2} \right) (t^2 - t). \nonumber
\end{align}
Observe that $G(1) = 0$ and 
\[
G'(1) = 2 I_p(1) - 2I_p(0) - 1/2 = - \log\left(\frac{p}{1 -p}\right) -1/2.
\]
So, $G'(1)>0$ if $h_p < -1/2$, i.e., if $p < \frac{e^{-1/2}}{1+e^{-1/2}}$. So, for $t$ sufficiently close to $1$, we have $\mathcal{V}(g_t) < \mathcal{V}(f_t)$, and we conclude that \eqref{qn2lim3} holds with strict inequality. Hence, the importance sampling scheme associated with the edge tilt $\ProbQ_n^{h_t,0}$ cannot be asymptotically optimal. 

\end{proof}

\section{Characterizing regimes for the triangle tilt}
\label{sec:characterizeRegimes}

Proposition \eqref{prop: asymptotic optimality} describes the necessary and sufficient condition \eqref{trianglesufficient} on the parameters $(\beta,\alpha)$ of a triangle tilt, that will produce an optimal importance sampling scheme. Given $(p,t)$, do these optimal tilt parameters $(\beta,\alpha)$ exist and how can they be found? In this section, we describe in a pseudo-explicit procedure for determining the optimal tilt parameters given $(p,t)$.

%------- What happens in the replica symmetry phase

An explicit determination of the optimal tilt parameters can be made when $(p,t)$ belongs to the replica symmetry phase.

\begin{proposition} \label{prop: RSP beta alpha characterize}
  If $(p,t)$ belongs to the replica symmetry phase, then there exists some $\alpha \in [2/3,1]$ for which the triangle tilt with parameters $(h_p,\beta,\alpha)$ produces an optimal scheme, where $\beta$ satisfies
\begin{align} \label{eq: beta RSP}
  \beta = \frac{h_t-h_p}{\alpha t^{3\alpha-1}}.
\end{align}
\end{proposition}

It will turn out that if there exists some $\alpha\in[0,1]$ and some $\beta$ for which the triangle tilt produces an optimal scheme, then for any $\alpha'\in [0,\alpha]$, and an appropriate $\beta'$ depending on $\alpha'$, the triangle tilt with parameters $(h_p,\beta',\alpha')$ also produces an optimal scheme (see Lemma \ref{lem: (appdx) S_alpha ordering}).
Thus, in Prop \ref{prop: RSP beta alpha characterize}, we can always take $\alpha=2/3$.

%------- A more general characterization applicable to both RSP and RBP

It is also of interest to determine the tilt parameters when $(p,t)$ belong to the replica breaking phase.
Our next result, Prop \ref{prop: general beta alpha characterize}, states a more general characterization of the optimal tilt parameters that applies to both the replica symmetry and breaking phases. To state the result, we introduce the minorant condition.

%-------- Introduce the minorant condition

We shall say that $(p,t)$ satisfies the \emph{minorant condition with parameter $\alpha$} if the point $(t^{3\alpha}, \phi_p(t))$ lies on the convex minorant of the function $x\mapsto \phi_p(x^{1/3\alpha})$.
In this case, subdifferential(s) of the convex minorant of $x \mapsto \phi_p(x^{1/3\alpha})$ at $x=t^{3\alpha}$ always exist and are positive.
Recall that the subdifferentials of a convex function $f(x)$ at a point $x$ are the slopes of any line lying below $f(x)$ that is tangent to $f$ at $x$.
The set of subdifferentials of a convex function is non-empty; if the function is differentiable at $x$, then the set of subdifferentials contains exactly one point, the derivative $f'(x)$.

%-------- Minorant condition holds in RSP & subset of RBP

The minorant condition is not an unattainable one, as shown in the next lemma.

\begin{lemma} \label{lem: minorant cond attain}
  The parameters $(p,t)$ that satisfy the minorant condition with some $\alpha$ includes the replica symmetry phase as well as a non-empty subset of the replica breaking phase.
\end{lemma}

%-------- The general characterization

\begin{proposition} \label{prop: general beta alpha characterize}
  Suppose $(p,t)$ satisfies the minorant condition for some $\alpha\in[0,1]$. 
Then the triangle tilt with the parameters $(h_p,\beta,\alpha)$ produces an optimal scheme, where $\beta$ is such that $\frac{\beta}{6}$ is a subdifferential of the convex minorant of $x \mapsto \phi_p(x^{1/3\alpha})$ at $x=t^{3\alpha}$.

 Moreover, if $\phi_p(u)$ is differentiable at $t$, then
 \begin{equation} \label{beta-alpha characterize}
  \beta = \frac{2 \phi_p'(t)}{\alpha t^{3\alpha-1}}.
\end{equation}
\end{proposition}

Combining Lemma \ref{lem: minorant cond attain} and Prop \ref{prop: general beta alpha characterize}, there exists $(p,t)$ belonging to the replica breaking phase for which a triangle tilt that produces an optimal scheme exists. 
In particular, in the replica symmetry phase, since $\phi_p(t)=I_p(t)$ is differentiable at $t$, Prop \ref{prop: general beta alpha characterize} reduces to Prop \ref{prop: RSP beta alpha characterize}.
Thus, Prop \ref{prop: RSP beta alpha characterize} gives an explicit construction of the tilt parameters when $(p,t)$ belong to the replica symmetry phase.
In the replica breaking phase, we may need to resort to numerical strategies to find the tilt parameters.
Nonetheless, we emphasize that it is possible in principle to construct an optimal importance sampling estimator in the replica breaking phase even if the limiting behaviour of the conditioned graph is not known exactly.

\subsection*{Proofs of results.}

Notice that if $t$ attains the supremum in \eqref{trianglesufficient}, we may rewrite the condition as
\begin{align*} \label{eq: t=arg sup var form 2}
  t = \argsup{0\leq u\leq1} \left[\frac{\beta}{6}u^{3\alpha} - \phi_p(u) \right].
\end{align*}
Together with Prop \ref{prop: asymptotic optimality}, the next lemma immediately implies Prop \ref{prop: general beta alpha characterize}.

%----- Minorant condition implies argsup of variational form

\begin{lemma} \label{lem: minorant condition}
  Suppose $(p,t)$ satisfies the minorant condition for some $\alpha>0$.
Let $\beta$ be such that $\frac{\beta}{6}$ is a subdifferential of the convex minorant of $x \mapsto \phi_p(x^{1/3\alpha})$ at $x=t^{3\alpha}$.
Then $\sup_{0\leq u\leq1} [ \frac{\beta}{6} u^{3\alpha} - \phi_p(u) ]$ is maximized at $t$. 

  Moreover, if $\phi_p(u)$ is differentiable at $t$, then $\beta$ is defined in \eqref{beta-alpha characterize}.
\end{lemma}

%----- Proof of Minorant condition implies argsup of variational form

\begin{proof}
  The proof follows a similar technique to \cite{LubZhao12}.
  Using the rescaling $u \mapsto x^{1/3\alpha}$, the variational form $\sup_{0\leq u\leq1} [ \frac{\beta}{6} u^{3\alpha} - \phi_p(u) ]$ can be rewritten as
\[
\sup_{0\leq x\leq1} [ \frac{\beta}{6} x - \phi_p(x^{1/3\alpha}) ].
\]
Let $\hat{\phi}_p(x)$ denote the convex minorant of $x \mapsto \phi_p(x^{1/3\alpha})$.
The assumption that $\frac{\beta}{6}$ is a subdifferential of $\hat{\phi}_p(x)$ at $x=t^{3\alpha}$ implies that the maximum of $\sup_{x} [ \frac{\beta}{6} x - \hat{\phi}_p(x) ]$ is attained at $t^{3\alpha}$.
By the hypothesis of the lemma, we have assumed that $(p,t)$ satisfies the minorant condition for $\alpha$, so that the point $(t^{3\alpha},\phi_p(t))$ lies on $\hat{\phi}_p(x)$. Thus we have that $\hat{\phi}_p(t^{3\alpha}) = \phi_p(t)$ and so the maximum of $\sup_{x} [ \frac{\beta}{6} x - \phi_p(x^{1/3\alpha}) ]$ is also attained at $t^{3\alpha}$.
It follows that the maximum of $\sup_{u} [ \frac{\beta}{6} u^{3\alpha} - \phi_p(u) ]$ is attained at $t$.
(However, this maximum may not be unique. If the subtangent line defined by the subdifferential $\frac{\beta}{6}$ touches $\hat{\phi}_p$ at another point $r^{3\alpha}$, then $r$ also a maximum.)

To prove the last part of the lemma, if $\phi_p(u)$ is differentiable at $t$, then the subdifferential is simply the derivative. Then we have 
\[
0 = \frac{\partial}{\partial x}\Big|_{x=t^{3\alpha}} 
[ \frac{\beta}{6} x - \phi_p(x^{1/3\alpha}) ] 
= \frac{\beta}{6} - \phi_p'(t) \frac{t^{1-3\alpha}}{3\alpha}
\]
implies that $\beta = \frac{2 \phi_p'(t)}{\alpha t^{3\alpha-1}}$.
\end{proof}

%------ Proof of Lemma \ref{minorant cond attain}

\begin{proof}
(Proof of Lemma \ref{lem: minorant cond attain}.)
Recalling Definition \ref{def:replica symmetric} of the replica symmetric phase, the it follows from the arguments in \cite{LubZhao12} and Theorem 4.3 in \cite{ChaVar11} that any $(p,t)$ that belongs to the replica symmetric phase satisfies the minorant condition for some $\alpha\in[2/3,1]$.

We now show that there exists $(p,t)$ belonging to the replica breaking phase that satisfies the minorant condition for some $\alpha$.
Notice from Appendix \ref{sec: (appdx) S_alpha} and Figure \ref{fig:replicaphase} that there exists some a critical value $p_{crit}$ such that when $p\geq p_{crit}$, $(p,t)$ is replica symmetric for all $t \in [p,1]$; whereas when $p\leq p_{crit}$, there exists an interval $[\underline{r}_p,\overline{r}_p] \subset (p,1)$ where $(p,t)$ is replica breaking if $t \in [\underline{r}_p,\overline{r}_p]$, and $(p,t)$ is replica symmetric for all other values of $t$.

To see this, consider $\alpha=1/3$ and convex minorant of $x \mapsto \phi_p(x^{1/3\alpha}) = \phi_p(x)$.
For each $p<p_{crit}$, there exists an interval $[\underline{r}_p,\overline{r}_p] \subset (p,1)$ where $(p,t)$ is replica breaking if $t \in [\underline{r}_p,\overline{r}_p]$, and $(p,t)$ is replica symmetric for the other values of $t$.
Since $\phi_p(t)< I_p(t)$ if $t \in [\underline{r}_p,\overline{r}_p]$ and $\phi_p(t)= I_p(t)$ for other values of $t$, and since $I_p(u)$ is convex, the convex minorant of $\phi_p(x)$ must touch $\phi_p$ at at least one $t_p \in [\underline{r}_p,\overline{r}_p]$.
So $(p,t_p)$ is replica breaking and satisfies the minorant condition.
\end{proof}

%%%%%%%%%%%%%%%%%%%%%%%%%%%%%%%%%%%%%%%%%%%%%%%%%%%%%%%%%%%%%%%%%%%
%%%%%  %%%%%%  %%%%%%  %%%%%%  %%%%%%  %%%%%%  %%%%%%  %%%%%%  %%%%
%%%%%%%%%%%%%%%%%%%%%%%%%%%%%%%%%%%%%%%%%%%%%%%%%%%%%%%%%%%%%%%%%%%

\section{Numerical simulations using importance sampling} \label{sec:NumSim}

We implement the importance sampling schemes to show the optimality properties of the Gibbs measure tilts in practice.
Although we have thus far been considering importance sampling schemes that draw i.i.d. samples from the tilted measure $\ProbQ$, in practice it is very difficult to sample independent copies of exponential random graphs.
This is because of the dependencies of the edges in the exponential random graph, unlike the situation with an \ER graph where the edges are independent.
Thus, to implement the importance sampling scheme, we turn to a Markov chain Monte Carlo method known as the Glauber dynamics to generate samples from the exponential random graph.
The Glauber dynamics refers to a Markov chain whose stationary distribution is the Gibbs measure $\ProbQ_{n}^{h,\beta,\alpha}$.
The samples $\tilde{X}_k$ from the Glauber dynamics are used to form the importance sampling estimator $\tilde{M}_K$ in \eqref{eq: IS estimator}.
The variance of $\tilde{M}_K$ clearly also depends on the correlation between the successive samples.
However, in this paper, rather than focus on the effect of correlation on the variance of $\tilde{M}_K$, we instead investigate and compare the optimality of the importance sampling schemes, and show that importance sampling is a viable method for moderate values of $n$.

\subsection*{Glauber dynamics.}
For the exponential random graph $\calG_{n}^{h,\beta,\alpha}$, the Glauber dynamics proceeds as follows.

Suppose we have a graph $X=(X_{ij})_{1\leq i<j\leq n}$. The graph $\tilde{X}$ is generated from $X$ via the following procedure.

\begin{enumerate}[1.]
  \item Choose an edge $X_{ij}$, for some $(i,j)$, from $X$ uniformly at random. 
  
  \item For the new graph $\tilde{X}$, fix all other edges $\tilde{X}_{i'j'} = X_{i'j'}$, for $(i',j')\neq(i,j)$.
  
  \item Conditioned on all other edges fixed, pick 
  \[
  \tilde{X}_{ij} \sim \text{Bern} (\varphi)
  \]
  where 
  \[
  \varphi = \frac{e^{h + (\beta/n) (L_{ij} + M_{ij})^\alpha (n^3/6)^{1-\alpha}}}
  {e^{h + (\beta/n) (L_{ij} + M_{ij})^\alpha (n^3/6)^{1-\alpha}} + e^{(\beta/n) M_{ij}^\alpha (n^3/6)^{1-\alpha}}}
  \]
  and where
  \[
  L_{ij} = \sum_{k\neq i,j} X_{ik} X_{jk} , \qquad \text{and} \qquad
  M_{ij} = \sum_{(k,l,m)\nsupseteq(i,j)} X_{kl} X_{km} X_{lm},
  \]
  is the number of 2-stars in $X$ with a base at the edge $X_{ij}$, and the number of triangles in $X$ not involving the edge $X_{ij}$, respectively.

  \item If conditioning on $A_J$ is used, check if $\tilde{X}$ is in $A_J$. If not, revert to $X$.
\end{enumerate}

In step 4, a conditioning of the Gibbs measure is discussed in Section \ref{sec:condGibbsmeas}.

For the classical exponential random graph with $\alpha=1$, the probability $\varphi$ in the Glauber dynamics has a neater expression,
\[
\varphi = \frac{e^{h+\beta L_{ij}/n}}{1 + e^{h + \beta L_{ij}/n}}.
\]

At each MCMC step, if $X_{ij} \neq \tilde{X}_{ij}$, then $E(\tilde{X})$ differs from $E(X)$ by one edge, and $T(\tilde{X})$ differs from $T(X)$ by $nL_{ij}$ triangles.
The stationary distribution of the Glauber dynamics is the Gibbs measure $\ProbQ_{n}^{h,\beta,\alpha}$ that defines the exponential random graph $\calG_{n}^{h,\beta,\alpha}$.
Regarding the mixing time of the Glauber dynamics, \cite{BhaBreSly08} showed for the case $\alpha=1$ that if the variational form for the free energy of the Gibbs measure $\ProbQ_{n}^{h,\beta,1}$,
\begin{align} \label{eq: var form alpha=1}
\sup_{0\leq u\leq1} [\frac{h}{2} u + \frac{\beta}{6} u^{3} - I(u)],
\end{align}
has a unique local maximum, then the mixing time of the Glauber dynamics is $\calO(n^2 \log n)$; otherwise, the variational form has multiple local maxima, and the mixing time is $\calO(e^n)$.

Clearly, the importance sampling tilt must be chosen so that the mixing time of the Glauber dynamics is $\calO(n^2 \log n)$.

%%%%%  %%%%%  %%%%%  %%%%%  %%%%%
\subsection{Example 1}
\label{subsec: Numerical results replica symmetric}

The importance sampling scheme was performed for $p=0.35, t=0.4$, in the replica symmetry phase.
We use the Glauber dynamics to draw samples from the edge tilt with parameters $h=h_t,\beta=0$, as well as from the triangle tilt with parameters $h=h_p, \alpha=1$ and $\beta = \frac{h_t-h_p}{t^2}$ as in \eqref{prop: RSP beta alpha characterize}.
The mixing time for both tilts is $\calO(n^2 \log n)$.
In addition to the edge and triangle tilts, we also consider a family of ``hybrid'' tilts with parameters $h=h_q$ for $q\geq p$, and $\alpha=1$ and 
\be
\beta = \beta_q = \frac{h_t-h_q}{t^2}. \label{hbetaRelate}
\ee
With these parameters, the variational form for the free energy of the corresponding Gibbs measure is uniquely maximized at $t$. 
Thus, the hybrid tilt satisfies the necessary condition in Proposition \ref{prop: ERu*NoOpt} for optimality (i.e., \eqref{eq: necessary cond for opt} does not hold).
By Theorem \ref{thm: Exponential graph free energy}, the corresponding exponential random graph $\calG_{n}^{h_q,\beta_q,1}$ is indistinguishable from the \ER graph $\calG_{n,t}$ and has a mean triangle density of $t^3$, in the sense of \eqref{QTmean2}.

In the simulations, we used the hybrid tilts with $h=h_q$, for $q=0.35, 0.36, \dots, 0.4$, and $\beta_q$ satisfying \eqref{hbetaRelate}.
With this notation, in fact, $q=p=0.35$ corresponds to the triangle tilt while $q=t=0.4$ corresponds to the edge tilt.
Table \ref{tab: HT beta_st qhat} verifies the accuracy of the importance sampling estimates for $\mu_{n} := \ProbP(\calG_{n,p} \in \calW_t)$ using the tilts $\ProbQ_{n}^{h_q,\beta_q,1}$.
Also shown is the estimate for the log probability, $\frac{1}{n^2} \log  \ProbP(\calG_{n,p} \in \calW_t)$, which can be seen to approach the LDP rate 
\begin{equation*}
  \lim_{n\rightarrow \infty} \frac{1}{n^2} \log  \ProbP(\calG_{n,p} \in \calW_t) = - I_{p}(t) \approx - 0.002694.
\end{equation*}
as $n$ is increased.
Table \ref{tab: HT beta_st qhatvar} shows the estimated values of the variance of the estimator, $Var_{\ProbQ_n}(\hat{q}_n)$, where $\hat{q}_n = {\bf1}_{\calW_t} \frac{d\ProbP_{n,p}}{d\ProbQ_{n}^{h,\beta}}$, as well as the log second moment $\frac{1}{n^2} \log \Mean_{\ProbQ_n} [\hat{q}_n^2]$.
The variance of the estimator for all the hybrid and edge tilts appear to be comparable to the optimal triangle tilt, and the log second moment likewise appears to converge towards $ - 2I_p(t) \approx -0.0053869$.
Notice that the parameters $p=0.35,t=0.4$ do not satisfy the hypothesis of Proposition \ref{prop:edgeNoOpt}, so the assertion of non-optimality of the edge tilt may not apply in this case. Regardless, non-optimality is not apparent for mid-sized graphs up to $n=96$.

For $n=16,32,64$, the number of MCMC samples used was $5\times 10^4 \,n^2 \log n$, while for $n=96$, the number of MCMC samples used was $10^5 \,n^2 \log n$.

\begin{table}
\small\centering
  \begin{tabular}{|c|c|c|c|c|c|c|}
  \hline
  \hfill $q$ &0.35&0.36&0.37&0.38&0.39&0.4   \\
         $n$ &    &    &    &    &    &      \\ \hline
16&0.12475&0.1247&0.12521&0.12425&0.12441&0.12435 \\
  &(-0.008131) &(-0.008132)&(-0.008116)&(-0.008146)&(-0.008141)&(-0.008143) \\ \hline
32&0.01107&0.011056&0.011116&0.010941&0.010972&0.010729\\
  &(-0.004398) &(-0.004399)&(-0.004394)&(-0.004409)&(-0.004407)&(-0.004429)\\ \hline
64&2.1919e-06&2.0283e-06&2.6073e-06&5.3287e-07&1.3822e-06&3.5772e-06\\
  &(-0.003181) &(-0.003200)&(-0.003139)&(-0.003527)&(-0.003294)&(-0.003062)\\ \hline
96&1.1036e-11&1.6868e-11&2.0805e-11&4.4039e-11&2.6124e-11&4.497e-11\\
  &(-0.002738) &(-0.002692)&(-0.002669)&(-0.002587)&(-0.002644)&(-0.002585)\\ \hline
  \end{tabular}
  \caption{Comparison of the estimates for the probability $\mu_{n}$ (top number) for varying tilts with parameters $(h_q,\beta,1)$, where $\beta$ is defined in \eqref{hbetaRelate}. Also shown is the log probability $\frac{1}{n^2} \log \ProbP(\calG_{n,p} \in \calW_t)$ (lower number).} 
  \label{tab: HT beta_st qhat}
\end{table}

\begin{table}
\small\centering
  \begin{tabular}{|c|c|c|c|c|c|c|}
  \hline
$n\ \backslash\ q$ &0.35&0.36&0.37&0.38&0.39&0.4 \\ \hline
16&0.030839&0.03007&0.030173&0.030553&0.031902&0.034105\\
  &(-0.01199) &(-0.01206)&(-0.01204)&(-0.01203)&(-0.01191)&(-0.01173)\\ \hline
32&0.00039386&0.00038614&0.00040055&0.00042058&0.00047462&0.00052598\\
  &(-0.007391) &(-0.007407)&(-0.007377)&(-0.007347)&(-0.007253)&(-0.00718)\\ \hline
64&2.9982e-11&2.5716e-11&4.6804e-11&2.3144e-12&1.9783e-11&1.8035e-10\\
  &(-0.005879) &(-0.005917)&(-0.005774)&(-0.006513)&(-0.005995)&(-0.005461)\\ \hline
96&1.157e-21&2.7721e-21&4.9044e-21&2.8661e-20&1.4562e-20&6.8628e-20\\
  &(-0.005220) &(-0.005125)&(-0.005065)&(-0.004876)&(-0.004951)&(-0.004785)\\ \hline
  \end{tabular}
  \caption{Comparison of the estimates for the variance $Var_{\ProbQ} (\hat{q}_n)$ (top number) for varying tilts with parameters $(h_q,\beta,1)$, where $\beta$ is defined in \eqref{hbetaRelate}. Also shown is the log second moment $\frac{1}{n^2} \log \Mean_{\ProbQ}[\hat{q}_n^2]$ (lower number). }
  \label{tab: HT beta_st qhatvar}
\end{table}

Both the random graphs corresponding to the triangle or edge tilts are expected by \eqref{QTmean2}, \eqref{QEmean2} to have triangle density of $t^3$ and edge density of $t$, on average. 
However, there is a difference between the way that the triangle and edge tilts produce events in $\set{\calT(f) \geq t^3}$, which is that the edge tilt tends to produce more successful samples in $\set{\calT(f) \geq t^3}$ with higher edge density, compared to the triangle tilt. (See Figure \ref{fig: Edgehist}.) This is attributable to the fact that the edge tilt penalizes successful samples that contain the desired triangle density but with lower than expected edge density.

\begin{figure}
  \centering
  \includegraphics[width=.8\linewidth]{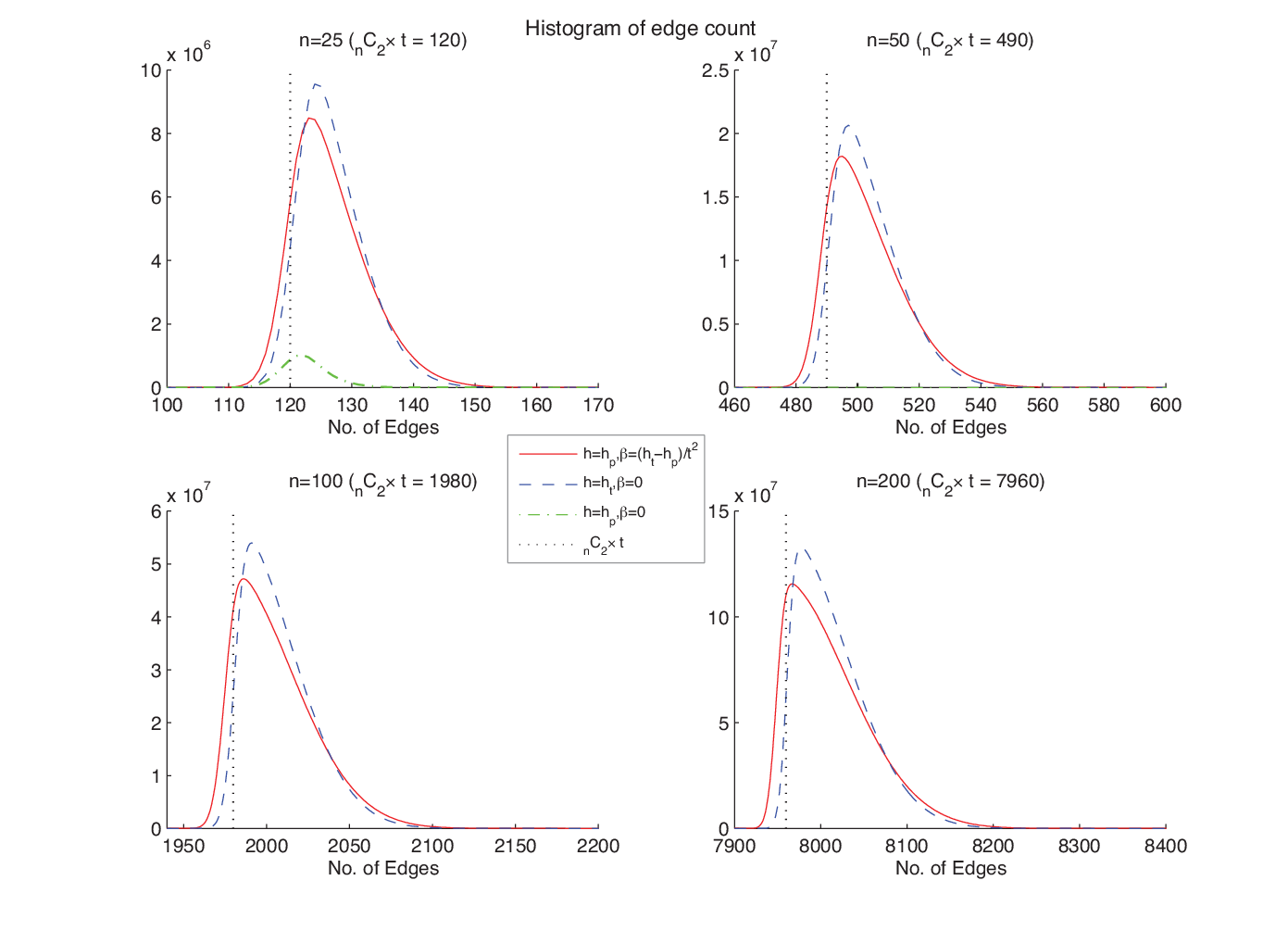}
  \caption{Histogram of the number of edges in the samples obtained using the importance sampling scheme based on the triangle tilt (solid red line) and edge tilts (dashed blue line), conditioned on the rare event $\{T(X)\geq \binom{n}{3}t^3\}$. The dotted green line in the top left panel shows the histogram for direct Monte Carlo sampling. The vertical line indicates the expected number of edges of the graph $\calG_{n,p}$ conditioned on the rare event.
}
  \label{fig: Edgehist}
\end{figure}

%%%%%  %%%%%  %%%%%  %%%%%  %%%%%  %%%%%  
\subsection{Example 2: Using $\alpha\neq1$ or conditioned Gibbs measures}
\label{sec:condGibbsmeas}

The importance sampling scheme was next performed for $p=0.2,t=0.3$, in the replica symmetric phase. We again use the Glauber dynamics to draw samples from the edge tilt with parameters $h=h_t,\beta=0$; the mixing time here is $\calO(n^2 \log n)$. In contrast to the previous example, for the triangle tilt with $\alpha=1$, the variational form \eqref{eq: var form alpha=1} has two local maxima, resulting in a mixing time of $\calO(e^n)$. Instead, we will use a triangle tilt with $\alpha=2/3$. Thanks to the fact that $(p,t)$ is in the replica symmetric phase and $\phi_p$ is differentiable at $t$, Proposition \ref{prop: general beta alpha characterize} implies that for $\alpha=2/3$, we choose
$$
\beta = \frac{h_t-h_p}{(2/3) t}.
$$
The simulation results for the importance sampling scheme using the triangle tilt with $\alpha=2/3$ is shown in Tables \ref{tab: mean estimate, conditioned and alpha tilts}, \ref{tab: var estimate, conditioned and alpha tilts} and compared with the results for the edge tilt. The simulation using direct Monte Carlo sampling is also shown for $n=32$. We see that the triangle tilt with $\alpha=2/3$ outperforms both the edge tilt and direct Monte Carlo simulation. Notice that the parameters $p=0.2,t=0.3$ do not satisfy the hypothesis of Proposition \ref{prop:edgeNoOpt}, but the edge tilt already appears to be non-optimal for mid-sized graphs up to $n=64$

Alternatively, we also consider a modification to the triangle tilt with $\alpha=1$ and $\beta_p = \frac{h_t-h_p}{t^2}$ as in \eqref{prop: RSP beta alpha characterize}. This modification draws samples from the Gibbs measure $\ProbQ_n^{h_p,\beta_p,1}$ conditioned on the event that the edge and triangle densities not exceed a given threshold $r$. To be specific, let
\begin{equation} \label{eq: A_r}
A_r = \{f\in\calW :\, \calT(f)\leq r^3 \text{ and } \calE(f) \leq r\}
\end{equation}
for some $r>t$, and let the conditioned triangle tilt be defined by the Gibbs measure conditioned on $A_r$,
\begin{align} \label{eq: conditioned Gibbs measure}
  \tilde{\ProbQ}_{n,A_r}^{h_p,\beta_p,1} (X) \propto \begin{cases}
    e^{n^2(\frac{h}{2}\calE(X) + \frac{\beta}{6}\calT(X))}, &\text{if }X\in A_r\\
    0 &\text{if } X\notin A_r
  \end{cases}.
\end{align}
In the numerical simulations, the threshold is chosen to be $r \approx 0.4272 >t$, which is a local minimum of the variational form \eqref{eq: var form alpha=1}. The motivation for this choice of threshold $r$ is discussed in the Appendix. The results for the conditioned triangle tilt are shown in Tables \ref{tab: mean estimate, conditioned and alpha tilts}, \ref{tab: var estimate, conditioned and alpha tilts}, which indicate that both triangle tilts perform comparably and both outperform the edge tilt.
%The number of MCMC samples taken was $10^5 n^2 \log(n)$ for all parameters, except for direct Monte Carlo, $q=s=p$, where $5 \times 10^5 n^2 \log(n)$ samples were taken.

\begin{table}
  \begin{tabular}{|c|c|c|c|c|}
  \hline
  $n$ & Triangle tilt $\alpha=2/3$ & Conditioned triangle tilt & Edge tilt & Monte Carlo \\\hline
  16 & 0.0064    & 0.006474   &  0.006285  &        \\
     & (-0.0197) & (-0.0197)  &  (-0.0198) & \\\hline
  32 & 4.3148e-7 & 3.5488e-7  & 3.3878e-7  & 3.7758e-7 \\
     & (-0.0143) & (-0.0145)  &  (-0.0145) & (-0.0144) \\\hline
  48 & 1.3976e-13& 1.1418e-14 & 1.2039e-12 &  ---      \\
     & (-0.0128) & (-0.0139)  &  (-0.0119) &  ---      \\\hline
  64 & 6.1882e-21& 2.9076e-23 & 1.8316e-19 &  ---      \\
     & (-0.0136) & (-0.0127)  & (-0.0105)  &  ---      \\\hline
  \end{tabular}
  \caption{Estimates for the probability $\mu_{n}$. In parenthesis is the estimator for the log probability $\frac{1}{n^2} \log \mu_n$.}
  \label{tab: mean estimate, conditioned and alpha tilts}
  \vskip10pt
  \begin{tabular}{|c|c|c|c|c|}
    \hline
  $n$ & Triangle tilt $\alpha=2/3$ & Conditioned triangle tilt & Edge tilt & Monte Carlo \\\hline
  16 & 1.5059e-4  & 2.4166e-4  &  1.0391e-3 &           \\
     & (-0.0334)  & (-0.0319)  &  (-0.0267) &           \\\hline
  32 & 1.5222e-12 & 1.9083e-12 & 6.7116e-11 & 3.7758e-7 \\
     & (-0.0265)  & (-0.0263)  &  (-0.0229) & (-0.0144) \\\hline
  48 & 2.6058e-25 & 3.268e-27  & 2.4737e-20 &  ---      \\
     & (-0.0245)  & (-0.0265)  &  (-0.0196) &  ---      \\\hline
  64 & 7.1703e-40 & 2.8806e-44 & 1.2806e-33 &  ---   \\
     & (-0.0220)  & (-0.0245)  &  (-0.0185) &  ---   \\\hline
  \end{tabular}
  \caption{Estimates for the variance $Var_{\ProbQ_{n}}(\hat{q}_n)$. In parenthesis is the estimate for the log second moment, $\frac{1}{n^2} \log \Mean_{\ProbQ_n} [\hat{q}_n^2]$.}
  \label{tab: var estimate, conditioned and alpha tilts}
\end{table}

\appendix

%%%%%%%%%%%%%%%%%%%%%%%%%%%%%%%%%%%%
%%%%%%%%%%%%%%%%%%%%%%%%%%%%%%%%%%%%
%%%%%%%%%%%%%%%%%%%%%%%%%%%%%%%%%%%%

%-------- S_\alpha

\section{Characterizing the phase diagrams}
\label{sec: (appdx) S_alpha}

We present in this appendix section a framework to define subregimes of the $(p,t)$ phase space, which extends the set up from \cite{LubZhao12}.

Recall that $(p,t)$ satisfies the \emph{minorant condition with parameter $\alpha$} if the point $(t^{3\alpha}, \phi_p(t))$ lies on the convex minorant of the function $x\mapsto \phi_p(x^{1/3\alpha})$.
Using the minorant condition and Lemma \ref{lem: minorant condition}, we define a parameterized family of subregimes of the $(p,t)$-phase space.

%-------- S_\alpha definition

\begin{definition} \label{def: S_alpha}
  Let $\alpha\in[0,1]$. We define the regime $\mathcal{S}_{\alpha}$ to be the set of parameters $(p,t)$ for which the minorant condition holds with $\alpha$.

  Further, we define $\mathcal{S}_{\alpha}^\circ \subset \mathcal{S}_{\alpha}$ to be the regime where, considering a subdifferential $\frac{\beta}{6}$ of the convex minorant of $x\mapsto \phi_p(x^{1/3\alpha})$, the variational form 
$$\sup_{0\leq u\leq1} [\frac{\beta}{6}u^{3\alpha} - \phi_p(u)]$$ 
is \emph{uniquely} maximized at $t$.
\qed
\end{definition}

%-------- S_{2/3}, RSP

Using the Definition \ref{def: S_alpha}, we can characterize the replica symmetry phase for conditioned \ER graphs, Definition \ref{def:replica symmetric}, in terms of $S_\alpha$. To this effect, the next lemma follows directly from Definition \ref{def:replica symmetric} using the arguments in \cite{LubZhao12} and \cite[Theorem 4.3]{ChaVar11}.

\begin{lemma} \label{lem: S_2/3}
  $\mathcal{S}_{2/3}^\circ$ is exactly the replica symmetric phase.
\end{lemma}

%---------- Uniqueness requirement

We highlight that in Definition \ref{def: S_alpha}, there exists a subdifferential $\frac{\beta}{6}$ such that the variational form is maximized at $t$, due to Lemma \ref{lem: minorant condition}, but $t$ may not be the unique maximizer; whereas the definition of $\mathcal{S}_\alpha^\circ$ requires that $t$ be the unique maximizer. The uniqueness requirement is convenient to make the clean connection with the replica symmetry phase for conditioned \ER graphs. As seen from Figure \ref{fig:replicaphase}, the replica symmetry phase $\mathcal{S}_{2/3}^\circ$ is the light and dark gray region to the right of the dotted curve, excluding the dotted curve. $\mathcal{S}_{2/3}$ includes the dotted curve.

%------- Figure of S_{2/3}

\begin{figure}
  \centering
  \includegraphics[width=.5\linewidth]{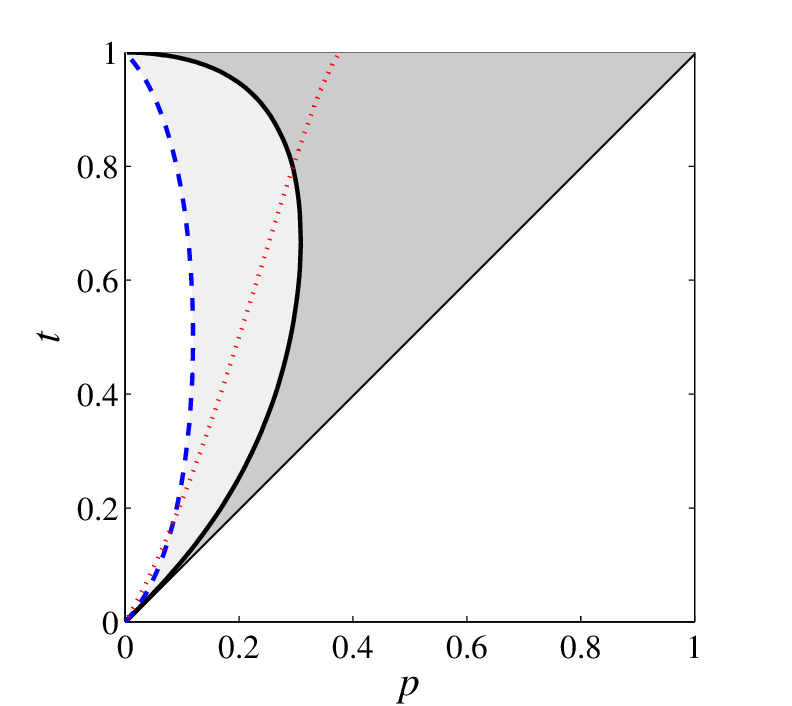}
  \caption{$\mathcal{S}_{1}^\circ$ is the dark gray region to the right of the solid curve (not including the solid curve). $\mathcal{S}_{2/3}^\circ$, the replica symmetric phase, is the light gray region to the right of the dashed curve (not including the dashed curve), together with the dark gray region. The diagonal red dotted line shows $G(t)=0$, where $G$ is the function in \eqref{eq: G(t)}. For parameters above this line, the edge tilt does not give an optimal importance sampling estimator. The straight line is $t = p$.}
  \label{fig:replicaphase}
\end{figure}

%-------- Rephrase minorant condition

Using the Definition \ref{def: S_alpha}, we can rephrase Prop \ref{prop: general beta alpha characterize} and Lemma \ref{lem: minorant cond attain} as follows.

\begin{corollary}
  If $(p,t) \in \mathcal{S}_\alpha$, then there exists a triangle tilt that produces an optimal scheme.

  Moreover, the regime $\bigcup_{\alpha>0} \mathcal{S}_{\alpha}$ where an optimal triangle tilt exists is strictly larger than the replica symmetric phase, and contains a nontrivial subset of the replica breaking phase. 
\end{corollary}

%------ S_\alpha ordering

Finally, we show in the next lemma that the union, $\bigcup_{\alpha>0} \mathcal{S}_{\alpha}$, is an increasing union as $\alpha\rightarrow0$.
A consequence of this lemma is that if $(p,t)\in \mathcal{S}_\alpha$ for some $\alpha\in[0,1]$, then for any $\alpha'\in [0,\alpha]$, the triangle tilt with parameters $(h_p,\beta',\alpha')$ also produces an optimal scheme, where $\beta'$ is appropriately chosen depending on $\alpha'$.

\begin{lemma} \label{lem: (appdx) S_alpha ordering}
  Let $\mathcal{S}_\alpha$ be defined in Definition \ref{def: S_alpha}. Then $\mathcal{S}_{\alpha'} \subset \mathcal{S}_{\alpha}$ for $0<\alpha<\alpha'$.
\end{lemma}

%------ Proof of S_\alpha ordering

\begin{proof} 
%\emph{(Proof of Lemma \ref{lem: (appdx) S_alpha ordering})}
  Denote $\phi_p^{\alpha}(x) = \phi_p(x^{1/3\alpha})$ and let $\hat{\phi}_p^{\alpha}(x)$ be the convex minorant of $\phi_p^{\alpha}(x)$. Then
\[
\phi_p^{\alpha'}(x) = \phi_p(x^{1/3\alpha'}) = \phi_p\big((x^{\alpha/\alpha'})^{1/3\alpha}\big) = \phi_p^{\alpha}(x^{\alpha/\alpha'}).
\]
Define $\eta(x) = \hat{\phi}_p^{\alpha'}(x^{\alpha'/\alpha})$. 
Let $K$ be the set where $\eta(x) = \phi_p^{\alpha'}(x^{\alpha'/\alpha})$ for $x\in K$.
Then 
\[
\eta(x) \leq \phi_p^{\alpha'}(x^{\alpha'/\alpha}) = \phi_p^{\alpha}(x)
\]
with equality occurring iff $x\in K$.
(The interpretation of $K$ is that $t^{3\alpha} \in K$ if and only if $(p,t)$ satisfies the minorant condition with $\alpha'$.)
Since $\frac{\alpha'}{\alpha} > 1$, the function $\eta(x)$ is convex and is less than $\phi_p^{\alpha}(x)$, hence it must be less than the convex minorant, $\eta(x) \leq \hat{\phi}_p^{\alpha}(x)$.
For $x\in K$, 
\[
\phi_p^{\alpha}(x) = \eta(x) \leq \hat{\phi}_p^{\alpha}(x) \leq \phi_p^{\alpha}(x)
\]
so $(x,\phi_p^{\alpha}(x))$ lies on the convex minorant $\hat{\phi}_p^{\alpha}(x)$ for all $x \in K$.
Hence, if $(p,t)$ satisfying the minorant condition with $\alpha'$, then $t^{3\alpha} \in K$ and $(t^{3\alpha},\phi_p(t))$ lies on the convex minorant $\hat{\phi}_p^{\alpha}(x)$, implying that $(p,t)$ satisfies the minorant condition with $\alpha$.

Now let $(p,t)$ satisfy the minorant condition with $\alpha'$, and suppose that $\frac{\beta'}{6}$ is a subdifferential of $\hat{\phi}_p^{\alpha'}(x)$ at the point $t^{3\alpha'}$ such that $\sup [\frac{\beta'}{6}u^{3\alpha} - \phi_p(u)]$ is uniquely maximized at $t$. According to the arguments in the proof of Lemma \ref{lem: minorant condition}, this means that the subtangent line 
\[
\ell_{\alpha'}(x) := \frac{\beta'}{6}(x-t^{3\alpha'}) - \phi_p(t)
\] 
lies below $\phi_p^{\alpha'}(x)$ and touches it at exactly one point $t^{3\alpha'}$. 
Let $\nu(x) = \ell_{\alpha'}(x^{\alpha'/\alpha})$. 
We have that $\nu(t^{3\alpha}) = \phi_p^{\alpha'}(t^{3\alpha'}) = \phi_p^{\alpha}(t^{3\alpha})$ and $\nu'(t^{3\alpha}) = \frac{\beta'}{6} \frac{\alpha'}{\alpha} t^{3(\alpha'-\alpha)}$.
Since $\frac{\alpha'}{\alpha}>1$, $\nu(x)$ is convex, and the line 
\[
\ell_{\alpha}(x) := \frac{\beta}{6}(x-t^{3\alpha}) - \phi_p(t),
\]
where $\frac{\beta}{6} =\nu'(t^{3\alpha})$, is tangent to $\nu(x)$ at the point $t^{3\alpha}$ and lies below $\nu(x)$.
For $x\neq t^{3\alpha}$, 
\[
\nu(x) = \ell_{\alpha'}(x^{\alpha'/\alpha}) < \phi_p^{\alpha'}(x^{\alpha'/\alpha}) = \phi_p^{\alpha}(x),
\]
so $\ell_{\alpha}(x)$ lies below $\phi_p^{\alpha}(x)$ and touches it at exactly one point $t^{3\alpha}$. Moreover, since $v(x)$ is a convex function less than $\phi_p^{\alpha}(x)$, we have $\hat{\phi}_p^{\alpha}(x) \geq \nu(x) \geq \ell_{\alpha}(x)$.
So, $\frac{\beta}{6}$ is a subdifferential of $\hat{\phi}_p^{\alpha}(x)$ and $\sup [\frac{\beta}{6}u^{3\alpha} - \phi_p(u)]$ is uniquely maximized at $t$.
The proof is complete.
\end{proof}

\begin{remark} \label{remark:param_edgeNoOpt}
In Prop \ref{prop:edgeNoOpt}, the critical value, $\tilde{p}= \frac{e^{-1/2}}{1+e^{-1/2}} \approx 0.3775$, corresponds to $h_{\tilde{p}} = -1/2$.
We see from Figure \ref{fig:replicaphase} that the conditions of the proposition are attained when if $p<\tilde{p}$ and $(p,t)$ is in the region above the red dotted line intersected with the replica symmetric phase. In this region, we have $G(t)<0$, where $G$ is defined at \eqref{eq: G(t)}. The edge tilt $\ProbQ^{h_t,0}_n$ does not produce an optimal estimator for the parameters in this region.
\end{remark}

%%%%%%%%%%%%%%%%%%%%%%%%%%%%%%%%%%%%%%%%%%%%%%%%%%%%%%%%%
%%%%%%%% Conditioned Gibbs measure

\section{Sampling from a conditioned Gibbs measure}

For exponential random graphs with $\alpha=1$, the Glauber dynamics is known to have an exponential mixing time $\mathcal{O}(e^{n})$ when the variational form \eqref{eq: var form alpha=1} has multiple local maxima \cite{BhaBreSly08}.
When considering a triangle tilt whose variational form has multiple local maxima, the slow mixing is one reason to preclude its feasibility as an importance sampling tilt. Another reason to avoid this tilt is because the global maximum of the variational form may not occur at $t$, even though its have a local maximum at $t$ by definition. Due to the second reason, such a tilt may produce a large number of samples with an over- or under-abundance of triangles, where the triangle density is determined by the global maximum of the variational form, rather than by the local minimum at $t$. This leads to a poor estimator that is not optimal and has large variance.

We propose to circumvent these problems by modifying the triangle tilt so that the sampled graphs are restricted to the subregion of the state space that has just the `right' number of triangles.
%However, conditioning the state space leads to a biased importance sampling estimator, so care must be taken to ensure the biased is controlled, or at least properly understood.

%%%%%%%%%%%%%%%%%%%%%%%%%%%%%%%%%%%%%%%%%%%%%
%%%%%%%%%%%%%%%%%%%%%%%%%%%%%%%%%%%%%%%%%%%%%
\subsection*{Conditioned Gibbs measure}
Given a set $A\subset \calW$, the exponential random graph conditioned on $A$, denoted $\calG_{n,A}^{h,\beta,\alpha}$ has the conditional Gibbs measure
\begin{align*}
  \tilde{\ProbQ}_{n,A}^{h,\beta,\alpha}(X) \propto \begin{cases}
    e^{n^2\calH(X)}, &\text{if }X\in A\\
    0 &\text{if } X\notin A
  \end{cases}
\end{align*}
where the Hamiltonian $\calH(X)$ is defined in \eqref{eq: Gibbs  measure}. The asymptotic behaviour of the free energy of the conditional Gibbs measure,
\begin{align*}
  \tilde{\psi}_{n,A}^{h,\beta,\alpha} = \frac{1}{n^2} \log \sum_{X\in A} e^{n^2\calH(X)},
\end{align*}
is described in the following proposition, which follows from a direct modification of \cite[Theorems 3.1, 3.2]{ChaDia11}.

\begin{proposition} \label{prop: Analogues of ChaDai theorems}
For any bounded continuous mapping $\calH: \calW \mapsto \Rm$, and any closed subset $A \subset \calW$, let $\tilde{\psi}_{n,A} = \tilde{\psi}_{n,A}^{h,\beta,\alpha}$ as above. 
Then 
    \begin{equation} \label{eq: A_J's free energy general Hamiltonian}
      \lim_{n\rightarrow\infty} \tilde{\psi}_{n,A} = \sup_{f\in A} [\calH(f) - \calI(f)].
    \end{equation}
    
    Moreover, if the variational form is maximized on the set $\tilde{\calF} \subset A$, then the corresponding conditioned exponential random graph is asymptotically indistinguishable from $\tilde{\calF}$.
\end{proposition}

As a consequence of Proposition \ref{prop: Analogues of ChaDai theorems}, an argument akin to the proof of Theorem \ref{thm: Exponential graph free energy} implies that for $(p,t)$ in the replica symmetric phase, the conditioned Gibbs measure \eqref{eq: conditioned Gibbs measure} conditioned on $A_r$ has free energy given by
\begin{align}
  \lim_{n\rightarrow\infty} \tilde{\psi}_{n,A_r} 
  = \sup_{0\leq u\leq r} \left[ \frac{h_p}{2} u + \frac{\beta}{6} u^3 - I(u) \right] .
\end{align}
By choosing $r$ so that the supremum is attained at $t$, the corresponding exponential random graph conditioned in $A_r$ is asymptotically indistinguishable from the \ER graph $\calG_{n,t}$. Thus, the necessary condition for asymptotic optimality of the importance sampling estimator is satisfied.

\subsection*{Importance sampling using the conditioned Gibbs measure}
The importance sampling scheme based on the conditioned Gibbs measure $\tilde{\ProbQ}_{n,A}^{h,\beta,\alpha}$ gives the estimator
\begin{equation} \label{eq: Q_Ar conditioned estimator}
  \hat{\nu}_{n} = 
  \frac{1}{K} \sum_{k=1}^K {\bf1}_{\calW_t} (\tilde{X_k}) 
  \frac{d\tilde{\ProbP}_{n,p,A_r}}{d\tilde{\ProbQ}_{n,A_r}^{h,\beta,\alpha}} (\tilde{X_k}),
\quad \text{where } \tilde{X}_k \sim \text{ i.i.d. } \tilde{\ProbQ}_{n,A_r}^{h,\beta,\alpha}
\end{equation}
where $\tilde{\ProbP}_{n,p,A_r}$ is the measure of the \ER graph conditioned on $A_r$.
Note that $\hat{\nu}_{n}$ is an unbiased estimator for $\nu_{n} = \tilde{\ProbP}_{n,p,A_r}(\calW_t)$, but it is a biased estimator for $\mu_n= \ProbP_{n,p}(\calW_t)$. 
The bias can be corrected by 
\begin{align*}
  \hat{\mu}_{n} = \hat{\nu}_{n} \cdot \ProbP_{n,p}(A_r) + \ProbP_{n,p}(\calW_t \cap A_r^c), 
\end{align*}
but the two probabilities on the RHS are not be easily computable or estimated.
Nonetheless, the choice of the set $A_r$ ensures the bias is small and vanishes exponentially faster than the small probability  $\mu_n$. In fact, standard computations give that $\ProbP_{n,p}(A_r) \rightarrow 1$ as $n\rightarrow\infty$, and
\begin{equation*}
\lim_{n\rightarrow\infty} \frac{1}{n^2} \log \ProbP_{n,p} \big( \calW_t \cap A_r^c \big) 
    \lneq - \inf_{f\in\calW_t} [I_p(f)] .
\end{equation*}

The asymptotic optimality of the importance sampling scheme is stated in the following result.

\begin{corollary} \label{cor: condition asymp optimal}
  Given $(p,t)$ in the replica symmetric phase, consider the conditioned triangle tilt defined by the Gibbs measure $\ProbQ_{n,A_r}^{h_p,\beta_p,1}$ in \eqref{eq: conditioned Gibbs measure}, conditioned on $A_r$ in \eqref{eq: A_r} with $p<t<r$. The importance sampling scheme based on this conditioned triangle tilt is asymptotically optimal.
\end{corollary}

%%%%%  %%%%%  %%%%%  %%%%%  %%%%%
\subsection*{Choosing the set $A_r$.}

We motivate the choice of the set $A_r$ in the Example from Section \ref{sec:condGibbsmeas}, with $p=0.2, t=0.3$ in the replica symmetric phase.
We had noted that for the triangle tilt with Gibbs measure $\ProbQ_n^{h_p,\beta_p,1}$, the variational form $V(u) = \frac{h_p}{2} u + \frac{\beta_p}{6} u^3 - I(u)$ has multiple local maxima. This is illustrated in Figure \ref{fig: phasecurve} (inset), where $t=0.3$ is a local maximum but not a global maximum, whereas $u^* \approx 0.989$ is the global maximum.
Without conditioning, the exponential graph $\calG_{n}^{h_p,\beta_p,1}$ has a mean triangle density of $(u^*)^3$, much greater than the desired triangle density of $t^3$; moreover, successive samples in the Glauber dynamics take exponentially long time to move from the region with a high triangle density $(u^*)^3$ to the region with a lower triangle density $t^3$. The effect of conditioning on the set $A_r$ is to cap the triangle density at $r^3$, and ensure faster mixing of the Glauber dynamics. Thus, one convenient choice of $r$ is to take $r\approx0.4272$ to be the local minimum of $V(u)$ which separates the two local maxima. Then, $t$ is the unique global maximum on the interval $[0,r]$ and the conditioned Gibbs measure has a mean triangle density of $t^3$.

Conditioning the Gibbs measure leads to a significant reduction in the asymptotic log second moment of the importance sampling estimator. 
This reduction is best illustrated by considering, besides the triangle tilt itself, the family of Gibbs measures with $h=h_p,\alpha=1$ and varying $\beta>0$. 
Figure \ref{fig: phasecurve} illustrates that as $\beta$ is increased from 0 up to a transition point $\beta \approx 4.76$, the variational form $V(u;\beta) = \frac{h_p}{2}u + \frac{\beta}{6} u^{3\alpha} - I(u)$ has a global maximum within the range $[0.2,0.3]$. 
For $\beta> 4.76$, the global maximum jumps up into the range $[0.9,1]$, so that the exponential random graph transitions from a regime of low edge density to one of high edge density. 
Near to the transition point, there is a range of $\beta$ for which $V(u;\beta)$ has two local maxima and one local minimum. Observe from the figure inset that the triangle tilt with $\beta=\beta_p$ lies in this range. 
It is for this range of $\beta$ that applying the conditioned Gibbs measure will lead to a reduction in the asymptotic log second moment. 
Figure \ref{fig: conditioning bounds} shows the asymptotic log second moment, both with and without conditioning of the Gibbs measure.
For each $\beta$, the threshold $r$ is chosen as the local minimum of $V(u;\beta)$. We see that conditioning the Gibbs measure significantly reduces the asymptotic log second moment, and the conditioned triangle tilt is asymptotically optimal. This is corroborated by the numerical simulations presented in Section \ref{sec:condGibbsmeas}. 
In contrast, when no conditioning is performed, the IS estimator exhibits a sharp decline in performance when $\beta$ is increased beyond the transition point at $\beta \approx 4.76$.

\begin{figure}
  \centering
  \includegraphics[width=.6\linewidth]{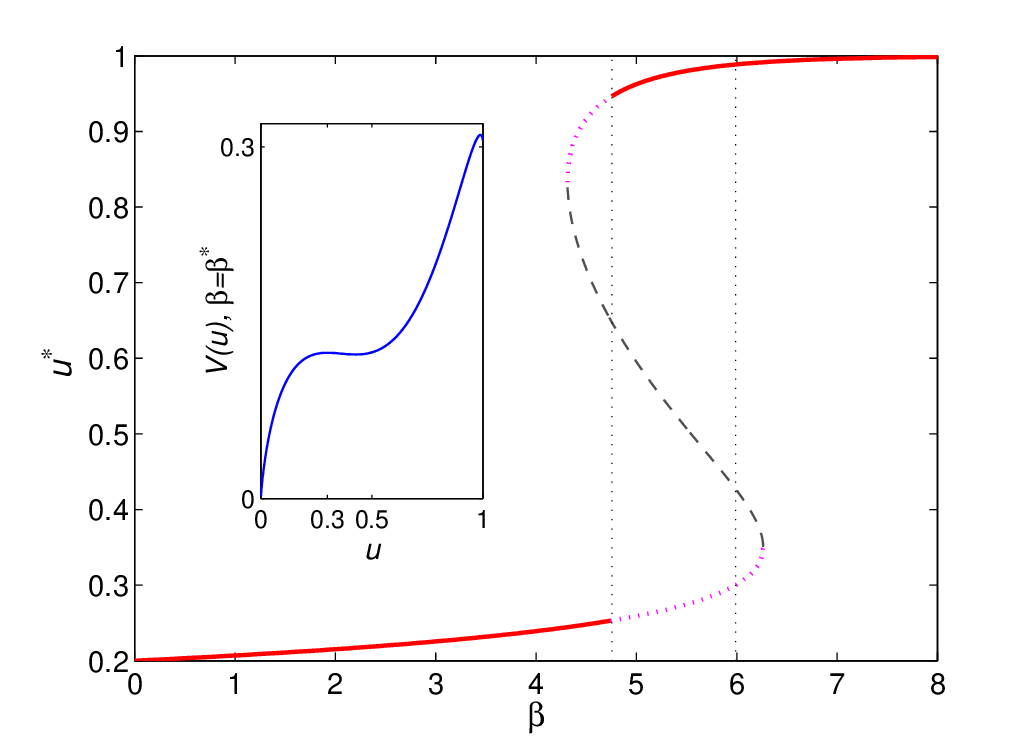}
  \caption{The phase curve denotes the values of the stationary points of the variational form $V(u) = \frac{h}{2}u + \frac{\beta}{6} u^{3\alpha} - I(u)$, as $\beta$ varies, and given $\alpha=1$, $h_p=\log\frac{p}{1-p}$, $p=0.2$. The red solid line denotes when the stationary point is a global maximum of $V(u)$; the red dotted line denotes the local maximum; the blue dashed line denotes the local minimum.
At the phase transition point at $\beta \approx 4.76$, the maximum of the variational form jumps from $u^*\approx0.253$ to $u^*\approx0.947$.
The inset shows the function $V(u)$ for $\beta = \beta^* \approx 5.99$ attaining a local maximum at $t=0.3$ and global maximum at $u^*\approx 0.989$.
}
  \label{fig: phasecurve}
\end{figure}

\begin{figure}
  \centering
  \includegraphics[width=.6\linewidth]{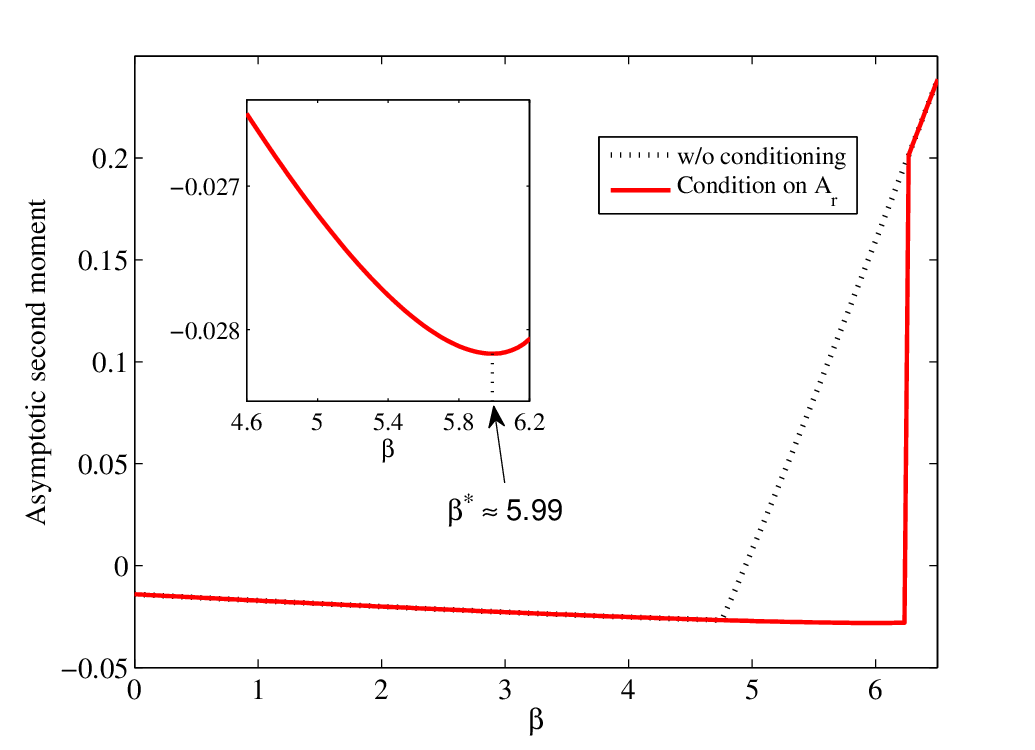}
  \caption{A plot of the asymptotic second moment, $\lim_{n\rightarrow\infty} \frac{1}{n^2} \log \Mean_{\tilde{\ProbQ}} [\hat{q}_{n,A_r}^2] $, of the importance sampling estimator based on the conditioned Gibbs tilt for fixed $h=h_p$ and varying $\beta$.
The insert is a zoom-in to show that the smallest variance is attained at $\beta=\beta^{\ast}$. 
The dotted line shows the rapid deterioration of the asymptotic second moment of the estimator without the use of conditioning.
Parameters used are $p=0.2$ and $t=0.3$.}
  \label{fig: conditioning bounds}
\end{figure}%
%

%%%%%%%%%%%%%%%%%%%%%%%%%%%%%%%%%%%%%%%%%%%%%%%%%%%%%%%%%
%%%%%%%% Auxiliary proofs

\section{Auxiliary lemmas and proofs}

We present a lemma on the asymptotic indistinguishability of an exponential random graph from a minimal set $\calF^*$, as well as the proof of Proposition \ref{prop: expo graph mean behaviour}.

\begin{lemma} \label{lem: minimalF^*}
  \begin{enumerate}[(i)]
    \item Given $(p,t)$, let $\calF^*$ be the set of functions that minimize the LDP rate function, $\inf_{f \in \calW_t} [\calI_p(f)]$ in \eqref{eq: LDP rate runction}. Then $\calF^*$ is the minimal set that the \ER graph $\calG_{n,p}$ conditioned on $\set{\calT(f) \geq t^3}$ is asymptotically indistinguishable from.

    \item Given $(h,\beta,\alpha)$, let $\calF^*$ be the set of functions that maximize $\sup_{f \in \calW} [\calH(f) - \calI(f)]$. Then $\calF^*$ is the minimal set that the exponential random graph $\calG_{n}^{h,\beta,\alpha}$ is asymptotically indistinguishable from.
  \end{enumerate}
\end{lemma}

\begin{proof}
The proofs of asymptotic indistinguishability of $\calF^*$ was shown in \cite[{Theorem 3.1}]{ChaVar11} for (i) and \cite[{Theorem 3.22}]{ChaDia11} for (ii).
The proofs naturally extend to give the minimality of $\calF^*$, and we state them here for the record.

Observe that for any random graph $\calG_{n}$ that is asymptotically indistinguishable from a set $\calF^*$, to show that $\calF^*$ is minimal, it suffices to show that, for any relatively open non-empty subset $\calF_0 \subset \calF^*$ such that $\calF^* \setminus \calF_0$ is non-empty, there exists $\epsilon>0$ such that 
\begin{equation} \label{eq: minimal F}
\liminf_{n \to \infty} \frac{1}{n^2} \log 
\ProbP( \delta_{\square}(\calG_{n}, \calF^*\setminus\calF_0) > \epsilon ) 
= 0.
\end{equation}

Let $\calF_0 \subset \calF^*$ be any relatively open non-empty subset, with $\calF^*\setminus \calF_0$ non-empty.
Denote, for $\eps > 0$, 
\begin{equation*}
  \calF_\eps = \set{ f \in \calW \; | \;\; \delta_\square(f, \calF^*\setminus\calF_0) > \eps }.
\end{equation*}

(i)
Since $\calF_0$ is relatively open in $\calF^*$, $\delta_\square (f, \calF^* \setminus \calF_0) > 0 $ for any $f \in \calF_0$.
So, there exists an $\eps>0$ sufficiently small such that $(\calF_\eps \cap \calW_t)^\circ$ contains at least one element of $\calF_0$.
($A^\circ$ denotes the interior of $A$.)
It follows that 
\begin{equation*}
  \inf_{f \in (\calF_\eps \cap \calW_t)^\circ} [\calI_p(f)]
  = \inf_{f \in \calW_t} [\calI_p(f)].
\end{equation*}
Since 
\begin{equation*}
  \ProbP(\calG_{n,p} \in \calF_\eps \,|\, \calG_{n,p} \in \calW_t) = \frac{\ProbP(\calG_{n,p} \in \calF_\eps \cap \calW_t)}{\ProbP(\calG_{n,p} \in \calW_t)},
\end{equation*}
from the large deviation principle in \cite[{Theorem 2.3}]{ChaVar11} implies that
\begin{align*}
  &\liminf_{n\rightarrow\infty} \frac{1}{n^2} \log \ProbP(\calG_{n,p} \in \calF_\eps \,|\, \calG_{n,p} \in \calW_t) \\
  &\qquad = \liminf_{n\rightarrow\infty} 
  \frac{1}{n^2} \log \ProbP(\calG_{n,p} \in \calF_\eps \cap \calW_t) - 
  \frac{1}{n^2} \log \ProbP(\calG_{n,p} \in \calW_t) \\
  &\qquad \geq - \inf_{f \in (\calF_\eps \cap \calW_t)^\circ} [\calI_p(f)]
  + \inf_{f \in \calW_t} [\calI_p(f)] \\
  &\qquad = 0.
\end{align*}

(ii) 
Since $\calF_0$ is relatively open in $\calF^*$, there exists an $\eps>0$ sufficiently small such that $\calF_\eps^\circ$ contains at least one element of $\calF_0$, and
\begin{equation*}
  \inf_{f \in \calF_\eps^\circ} [\calH(f) - \calI(f)]
  = \inf_{f \in \calW} [\calH(f) - \calI(f)].
\end{equation*}

Since the Hamiltonian $\calH$ is bounded, for any $\eta>0$, there is a finite set $A \subset \Rm$ such that the intervals $\{(a,a+\eta), a \in A \}$ cover the range of $\calH$.
Let $\calF_\eps^a = \calF_\eps \cap \calH^{-1}([a,a+\eta])$, and let $\calF_\eps^{a,n} = \calF_\eps^a \cap \Omega_n$ be the functions corresponding to a simple finite graph.
Then
\begin{align*}
  \ProbP(\calG_{n} \in \calF_\eps)
  \geq
  \sum_{a\in A} e^{n^2(a-\psi_n)} |\calF_\eps^{a,n}|
  \geq
  e^{-n^2\psi_n} \sup_{a\in A} \left[ e^{n^2a} |\calF_\eps^{a,n}| \right]
\end{align*}
and
\begin{align*}
  \frac{1}{n^2} \log \ProbP(\calG_{n} \in \calF_\eps)
  &\geq
  - \psi_n + 
  \sup_{a \in A} [ a - \frac{1}{n^2} \log |\calF_\eps^{a,n}| ]  .
\end{align*}
By an observation in \cite[{Eqn. (3.4)}]{ChaDia11}, for any open set $U \subset \calW$, and $U_n = U \cap \Omega_n$,
\begin{equation*}
  \liminf_{n\rightarrow\infty} \frac{1}{n^2} \log |U_n| \geq
  - \inf_{f \in U} [\calI(f)]  .
\end{equation*}
Then, since 
\begin{equation*}
  \sup_{f\in \calF_\eps^a}  [\calH(f) - \calI(f)] \leq 
  \sup_{f\in \calF_\eps^a}  [a+ \eta - \calI(f)] =
  a+ \eta - \inf_{f\in \calF_\eps^a}  [\calI(f)]
\end{equation*}
we have that 
\begin{align*}
  \liminf_{n\rightarrow\infty} \frac{1}{n^2} \log \ProbP(\calG_{n} \in \calF_\eps)
  &\geq
  - \sup_{f \in \calW} [\calH(f) - \calI(f)] + \sup_{a \in A} [a -   \inf_{f \in (F_\eps^a)^\circ} [\calI(f)]] \\
  &\geq 
  - \sup_{f \in \calW} [\calH(f) - \calI(f)] + \sup_{a \in A} \sup_{f \in (F_\eps^a)^\circ} [\calH(f) - \calI(f)]  - \eta \\
  &\geq 
  - \sup_{f \in \calW} [\calH(f) - \calI(f)] + \sup_{f \in F_\eps^\circ} [\calH(f) - \calI(f)]  - \eta\\
  &= 0   .
\end{align*}

The proof is complete.
\end{proof}

\vskip10pt\noindent
\textbf{Proof of Proposition \ref{prop: expo graph mean behaviour}.}
\begin{proof}
Let $\epsilon_1> 0$ be arbitrary. As in Theorem \ref{thm: Exponential graph free energy}, let $\calF^*_{v^*}$ be the set of minimizers of $\inf_{f\in \partial \calW_{v^*}} [\calI_q(f)]$.
\begin{align*}
  &\Mean_{\ProbQ_n} |\calT(X) - (v^*)^3|\\
  &\qquad =
  \int_{\{\delta_\square (X, \calF^*_{v^*}) > \epsilon_1\}} |\calT(X) - (v^*)^3| \,d\ProbQ_n(X) +
  \int_{\{\delta_\square (X, \calF^*_{v^*}) \leq \epsilon_1\}} |\calT(X) - (v^*)^3| \,d\ProbQ_n(X) \\
  &\qquad = (I) + (II)
\end{align*}
(We have dropped the superscripts, $\ProbQ_{n} = \ProbQ_{n}^{h_q,\beta,\alpha}$.) We estimate the two terms.
To estimate $(I)$, by \cite[{Theorem 4.2}]{ChaDia11}, there exists $C,\epsilon_2>0$ such that for sufficiently large $n$
\begin{align*}
  \ProbQ_n (\delta_\square (X, \calF^*_{v^*}) > \epsilon_1)
  \leq C_2 e^{-n^2\epsilon_2}.
\end{align*}
Since $|\calT(X) - (v^*)^3| \leq 1$,
\begin{align*}
  (I) \leq
  \ProbQ_n (\delta_\square (X, \calF^*_{v^*}) > \epsilon_1)
  \leq C_2 e^{-n^2\epsilon_2}.
\end{align*} 
To estimate $(II)$, for any $X \in \{\delta_\square (X, \calF^*_{v^*}) \leq \epsilon_1\}$, let the function $f^*_X \in \calF^*_{v^*}$ be such that $\delta_\square (X, f^*_X) \leq \epsilon_1$. Note that $\calT(f^*_X) = (v^*)^3$ by definition.
By Lipschitz continuity of the mapping $f \mapsto \calT(f)$ under the cut distance metric $\delta_{\square}$ \cite[{Theorem 3.7}]{Borgsetal08},
\begin{equation*}
  |\calT(X) - (v^*)^3| = |\calT(X) - \calT(f^*_X)| 
  \leq C_1 \delta_\square(X, f^*_X) \leq C_1 \epsilon_1.
\end{equation*}
So
\begin{align*}
  (II)
  &= \int_{\{\delta_\square (X, \calF^*_{v^*}) \leq \epsilon_1\}} |\calT(X) - (v^*)^3| \,d\ProbQ_n(X) \\
  &\leq C_1 \epsilon_1 \ProbQ_n (\delta_\square (X, \calF^*_{v^*}) \leq \epsilon_1) \\
  &\leq C_1 \epsilon_1 .
\end{align*}
Hence, 
\begin{equation*}
  \lim_{n\rightarrow \infty} \Mean_{\ProbQ_n} |\calT(X) - (v^*)^3| 
  \leq \lim_{n\rightarrow \infty} C_2 e^{-n^2\epsilon_2} + C_1\epsilon_1 
  = C_1\epsilon_1 .
\end{equation*}
Since $\epsilon_1$ is arbitrary, \eqref{QTmean2} follows.

If $(q,v^*)$ belongs to the replica symmetric phase, we have by Theorem \ref{thm: Exponential graph free energy} that $\calF^*_{v^*}$ consists uniquely of the constant function $f^*(x,y) \equiv v^*$. 
Then since $\calE(f^*) = v^*$, the above proof follows identically to yield that 
\begin{equation*}
  \lim_{n\rightarrow \infty} \Mean_{\ProbQ_n} |\calE(X) - v^*| 
  \leq \lim_{n\rightarrow \infty} C_2 e^{-n^2\epsilon_2} + C\epsilon_1 
  = C \epsilon_1 .
\end{equation*}
\end{proof}

%--------------------------------------------------------------------------------------------------

\bibliographystyle{plain}
\bibliography{ISbib}

\end{document}